\RequirePackage{fix-cm}
\documentclass[smallextended]{svjour3} 
\usepackage{graphicx}

\usepackage{amsmath,amssymb,amsfonts,fancybox,shadow,color}
\usepackage{hyperref}
\usepackage{booktabs}
\usepackage{xcolor,bm}
\usepackage{enumitem} 
\usepackage{subcaption}
\usepackage{verbatim}
\usepackage{algorithm}
\usepackage{algpseudocode}
\usepackage[normalem]{ulem}

\smartqed 
\renewenvironment{proof}{{\it Proof.}}{\hfill\qedsymbol}
\newcommand{\qedsymbol}{$\square$}

\newcommand{\RR}{{\mathbb{R}}}

\newcommand{\tT} {\tilde{T}}

\def\be#1\ee{\begin{equation}#1\end{equation}}

\newcommand\hgam{\boldsymbol{\hat\gamma}}
\newcommand\gam{\boldsymbol{\gamma}}

\newcommand{\Alpha}{\boldsymbol \alpha}
\newcommand{\Beta}{\boldsymbol \beta}

\newcommand{\hKappa}{\hat{\boldsymbol \kappa}}

\renewcommand{\Xi}{\boldsymbol \xi}
\newcommand\blF{ {\cal F}}
\newcommand\blf{ {f}}
\newcommand\blTF{ {\tilde {\cal F}}}
\newcommand\blC{{\cal C}}

\newcommand\blU{{\cal U}}

\newcommand\bx{\mathbf{x}}

\DeclareFontFamily{U}{mathx}{\hyphenchar\font45}
\DeclareFontShape{U}{mathx}{m}{n}{
 <5> <6> <7> <8> <9> <10>
 <10.95> <12> <14.4> <17.28> <20.74> <24.88>
 mathx10
 }{}
\DeclareSymbolFont{mathx}{U}{mathx}{m}{n}
\DeclareFontSubstitution{U}{mathx}{m}{n}
\DeclareMathAccent{\widecheck}{0}{mathx}{"71}
\DeclareMathAccent{\wideparen}{0}{mathx}{"75}

\definecolor{red}{rgb}{0,0,0}
\definecolor{purple}{rgb}{0,0,0}

\begin{document}

\title{Monotonicity, bounds and acceleration of Block Gauss and Gauss-Radau quadrature for computing $B^T \phi(A) B$} 

\titlerunning{Gauss-Radau quadrature for computing $B^T \phi(A) B$ } 

\author{J\"orn Zimmerling \and
 Vladimir Druskin \and
Valeria Simoncini
}


\institute{J\"orn Zimmerling, \at
 Institutionen f\"or informationsteknologi, Uppsala Universitet, L\"agerhyddsv\"agen 2, 75237 Uppsala, Sweden. \\
 \email{jorn.zimmerling@it.uu.se} 
 \and
 Vladimir Druskin, \at
 Worcester Polytechnic Institute, USA.
		\email{vdruskin@wpi.edu} 
		\and
		Valeria Simoncini,\\
		Dipartimento di Matematica and (AM)$^2$,
		Alma Mater Studiorum - Università di Bologna, 40126 Bologna, and IMATI-CNR, Pavia, Italy.\\
		\email{valeria.simoncini@unibo.it}
}

\date{\today}

\maketitle

\begin{abstract}
In this paper, we explore quadratures for the evaluation of $B^T \phi(A) B$ where $A$ is a symmetric positive-definite (s.p.d.) matrix in $\mathbb{R}^{n \times n}$, $B$ is a tall matrix in $\mathbb{R}^{n \times p}$, and $\phi(\cdot)$ represents a matrix function that is regular enough in the neighborhood of $A$'s spectrum, e.g., a Stieltjes or exponential function. These formulations, for example, commonly arise in the computation of multiple-input multiple-output (MIMO) transfer functions for diffusion PDEs.

We propose an approximation scheme for $B^T \phi(A) B$ leveraging the block Lanczos algorithm and its equivalent representation through Stieltjes matrix continued fractions. We extend the notion of Gauss-Radau quadrature to the block case, facilitating the derivation of easily computable error bounds. For problems stemming from the discretization of self-adjoint operators with a continuous spectrum, we obtain sharp estimates grounded in potential theory for Padé approximations and justify averaging algorithms at no added computational cost. The obtained results are illustrated on large-scale examples of 2D diffusion and 3D Maxwell's equations and a graph from the SNAP repository. We also present promising experimental results on convergence acceleration via random enrichment of the initial block $B$.

\end{abstract}

\section{Introduction}
\subsection{Problem statement}
Let $A=A^T\in\RR^{n\times n}$ be a symmetric positive definite matrix and let $B$ be a tall matrix with $B\in\RR^{n\times p}$, with $p\ll n$. We refer to $p$ as the block size. We want to approximate 
\be
\blF(s) =B^T\phi(A,s)B
\ee
via approximations of $\phi(A,s)B$ on a block Krylov subspace 
\[
\mathcal{K}_m(A,B)={\rm blkspan} \{B, AB, A^2 B , \dots A^{m-1} B \}
\]
where blkspan means that the whole space range$([B, AB, A^2 B , \dots A^{m-1} B])$ is generated.
{
In the literature various functions $\phi$ are considered to derive bounds for the scalar case, for instance strictly completely monotonic functions \cite{BenziMatrixF,GaussRadauBoundNetwork} or functions with sufficient derivatives on the spectral interval of $A$ \cite{lot2008}.} Here, we mainly consider the simpler problem $\phi(A,s)=(A+sI)^{-1}$ with $s\in\RR_+$,
which can be generalized to Stieltjes functions. Our numerical experiments will also 
illustrate the applicability of our results for exponentials and resolvents 
with purely imaginary $s\in\imath\RR$, which is an important class of problems in the computation of transfer functions of linear time-invariant (LTI) systems with imaginary Laplace frequencies
and time domain responses.
To simplify notation, without loss of generality, we assume that the matrix $B$ has orthonormal columns.

Gaussian-quadrature rules via Krylov subspace methods \cite{GM10} have been extended to the block case \cite{RRT16,lot2013}. Gauss-Radau quadrature rules have also been extended to the block case via the block Krylov subspace \cite{Lun18}, where at each block Krylov iteration, $p$ Ritz values of $A$ projected onto the block Krylov subspace are prescribed by a low-rank modification of this projection. These $p$ Ritz values are distinct and {\it larger} than the maximum eigenvalue of $A$, which did not lead to an acceleration in convergence of the block Krylov method in~\cite{Lun18}.

 In this paper, we take a different approach and extend Gauss-Radau quadrature to the block case by prescribing $p$ identical Ritz values that are {\it smaller} than the minimum eigenvalue of $A$. When considering the quadratic form $B^T \phi(A) B$, this allows us to obtain two-sided estimates of the solution at the cost of the standard block Lanczos algorithm in the Krylov block subspace $\mathcal{K}_m(A,B)$. We further show that this choice leads to a (convergent) monotonically decreasing approximation and allows for computable and tight error bounds. This definition is consistent with the Gauss-Radau bound for the $A$-norm of the error in CG, which to date only has been developed for the scalar case \cite{Meurant2023}. The matrices we consider in this work arise as discretizations of PDE operators in unbounded domains whose smallest eigenvalue tends to zero, therefore, the $p$ prescribed Ritz values are zero.

We focus on applications stemming from diffusive electromagnetic problems in an unbounded domain where we show that averaging Gauss-Radau and Gauss quadratures leads to a cost-free reduction in approximation error. This significantly improves the accuracy of the standard Krylov projection approximation, even though matrix-function approximation by Krylov projection is near optimal in global norms \cite{amsel2023nearoptimal}.

A vast body of literature exists for approximating $\phi(A)B$ via block Krylov subspace methods for symmetric and non-symmetric matrices $A$ \cite{FLS17}. Ideas based on Gauss-Radau quadratures have motivated the Radau-Lanczos method for the scalar (non-block) case \cite{FLSS17}, and the Radau-Arnoldi method for the non-symmetric block case \cite{FLS20}. Similarly to the Gauss-Radau definition in \cite{Lun18}, the Radau-Arnoldi method in \cite{FLS20} modifies $p$ Ritz values at each iteration to be larger than the largest eigenvalue of $A$. The considered functions $\phi(\cdot)$ in these works are Riemann-Stieltjes integrals, and we note that the results obtained for the resolvent in this work carry over straightforwardly to such functions.

Monotonicity results for the ordinary Lanczos algorithm were obtained for the exponential function in \cite{Druskin2008}; { see also similar results in \cite{Frommer2009}. Two-sided monotonic bounds for scalar Gauss and Gauss-Radau quadratures were first obtained in \cite{lot2008}. In this work we generalize these results to {matrix-valued} quadratures calculated via the block Lanczos algorithm}. For instance, using pairs of Gauss and Gauss-Radau quadratures has been explored for functions of adjacency matrices in \cite{GaussRadauBoundNetwork}. For the related conjugate gradient algorithm, Gauss-Radau based upper bounds for the scalar case have been developed in \cite{Meurant2023,MeurantBook}. For rational functions of matrices, error bounds for the classical Lanczos algorithms have been obtained in \cite{SimonciniErrorBound}. Recently, error bounds for approximating $\phi(A)B$ with the Lanczos method have been extended to the block Lanczos method \cite{Chen24}. A connection of the block Lanczos algorithm and matrix continued fraction similar to the one used in this work was described in the context of quantum physics for MIMO transfer function approximations in \cite{Jones1987TheRM}.

The remainder of the article is organized as follows: we review the basic properties of the block Lanczos algorithm in section~\ref{sec:BlLanc}, and we show that the resolvent of the Lanczos block tridiagonal matrix leads to a Stieltjes-matrix continued fraction in section~\ref{sec:GaussSfrac}. Using these continued fractions we define the notion of Gauss-Radau quadrature for the block case in section~\ref{sec:GaussRadau}. This is used in section~\ref{sec:ErrBound} to derive a two-sided error bound. Convergence acceleration based on averaged Gauss-Radau and Gauss quadrature is presented in the same section with a more detailed treatment in Appendix~\ref{sec:PostQuadErr}. Last, numerical examples are presented in section~\ref{sec:NumEx}.

\subsection{Notation}
Given two square symmetric matrices $G_1, G_2$ we use the notation
$G_1<G_2$ to mean that
the matrix $G_2-G_1$ is positive definite. 
A sequence of matrices $\{G_m\}_{m\ge 0}$ is said to be monotonically
increasing (resp. decreasing) if $G_m < G_{m+1}$ (resp. $G_{m+1}<G_m$)
for all $m$.
Positive-definite $p \times p$ matrices are denoted by Greek 
letters $\Alpha,\Beta,\gam$ and matrix-valued functions by calligraphic capital letters such as $\blC(s)$ or $\blF(s)$. 
Last, for $\Alpha, \Beta \in \mathbb{R}^{p\times p}$ and $\Beta$ is nonsingular,
{ we use the notation $\frac{\Alpha}{\Beta}:= \Alpha \Beta^{-1}$ (right inversion).}
The matrix $E_k \in \mathbb{R}^{mp\times p}$ has zero elements except for the
$p\times p$ identity matrix in the $k$-th block, $E_k =[0,\ldots, 0,I, \ldots, 0]^T$.

\section{Block Gauss Quadratures via the Block Lanczos Algorithm}\label{sec:BlLanc}

{ For scalar $\blF(s)$, i.e. {$p=1$}, the approximation of $\blF(s)=B^T (A+ sI)^{-1}B$ by means of
$\blF_m(s)=E_1^T (T_m+ sI)^{-1}E_1$, where $T_m$ is a specifically designed s.p.d. tridiagonal matrix, can be classically interpreted as a quadrature formula for integrals \cite{GaussRadauBoundNetwork,lot2008}. The block setting with $p>1$ and $T_m$ a specifically designed s.p.d. block tridiagonal matrix allows for an analogous block quadrature interpretation. 

Consider the eigendecompositions $A=Z\Lambda Z^T$ with eigenpairs $(\lambda_i,z_i)$ and $T_m=\tilde Z_m \tilde\Lambda_m \tilde Z_m^T$, $T_m \in \mathbb{R}^{mp\times mp}$ with eigenpairs $(\tilde\lambda_i,\tilde z_i)$ . Then we can rewrite $\blF(s)$ as an integral 
$$
B^T (A+ sI)^{-1}B = \sum_{i=1}^N \frac{(z_i^T B)^T(z_i^T B)}{\lambda_i+s} = \int_{0}^\infty \frac{1}{\lambda+s} \,{\rm d}\mu(\lambda),
$$
with symmetric, matrix valued spectral measure
$$
 \mu(\lambda)= \sum_{i=1}^{N}(z_i^T B)^T(z_i^T B) H(\lambda-\lambda_i),
$$
with $H(\cdot)$ the Heaviside step function. The approximation $\blF_m(s)$ of this integral can be written as
\be
E_1^T (T_m+ sI)^{-1}E_1 = \sum_{i=1}^{mp} \frac{(\tilde z_i^T E_1)^T(\tilde z_i^T E_1)}{\tilde \lambda_i+s}, 
\ee
which can be interpreted as a quadrature rule $\sum_i w_i \frac{1}{x_i+s}$ with weights $w_i=(\tilde z_i^T E_1)^T(\tilde z_i^T E_1)$ and nodes $x_i=\tilde \lambda_i$. Such a quadrature rule is said to be a Gauss quadrature rule if it exactly integrates polynomials of degree (strictly) less than $2m$. A quadrature rule is called Gauss-Radau quadrature rule if additionally, one quadrature node is fixed at either end of the integration interval. In this paper we consider matrices $A$ stemming from the discretization of a differential operator with continuous spectrum approaching $0$, hence we will fix a quadrature node at $x_0/\tilde \lambda_0 = 0$ \cite{GM10}.

 Gauss-Radau quadrature has been extended to the block case in Section~4.3 of \cite{Lun18} where $p$ quadrature nodes are prescribed that are distinct and {\it larger} than the maximum eigenvalue of $A$. In this extension of the Gauss-Radau quadrature to the block case, we propose to prescribe the first quadrature point at $0$ with quadrature weight $w_i$ of rank $p$.}

 {\begin{center}
\begin{minipage}{.65\linewidth}
 \begin{algorithm}[H]
\caption{Block Lanczos iteration}\label{alg:blockLanc}
\begin{algorithmic}
\normalsize
\State Given $m$, $A\in{\mathbb R}^{n\times n}$ s.p.d., 
$B\in{\mathbb R}^{n\times p}$ with orthonormal columns 
\State $Q_1 =B$
\State $W = AQ_1$
\State $\Alpha_1 = Q_1^T W$
\State $W = W - Q_1 \Alpha_1$
\For{$i= 2,\dots, m$} 
 	\State $Q_i\Beta_{i}=W$ \hskip 0.5in [QR decomposition of $W$]
	\State $W = AQ_i- Q_{i-1}\Beta_i^T$
 	\State $\Alpha_i = Q_i^H W$
 	\State $W = W - Q_i \Alpha_i$
\EndFor 
\end{algorithmic}
 \end{algorithm}
\end{minipage}
\end{center}}
\vspace{0.5cm}

To obtain a (block)-symmetric $T_m$ and a quadrature approximation the block Lanczos algorithm (Algoritm~\ref{alg:blockLanc}) is used. Let us assume that we can perform $m$ steps, with $mp \le n$, of 
the block Lanczos iteration without breakdowns 
or deflation\footnote{Deflation in the block recurrence is required if for some $m$, the generated space has dimension strictly less than $mp$. In that case, the redundant linearly dependent columns generated by the block Lanczos iteration need to be purged, and the block size reduced accordingly, see, e.g., \cite{Cullum.Willoughby.85}. In the numerical experiment presented in section~\ref{sec:NumEx} breakdown did not occur.} \cite{O'Leary1980,GOLUB1977}. As a result, 
the orthonormal block vectors $Q_i\in \mathbb{R}^{n \times p}$ form 
the matrix ${\bm Q}_m=[Q_1, \ldots, Q_m]\in \mathbb{R}^{n\times mp}$, whose columns
contain an orthonormal basis for the block Krylov subspace
\[\mathcal{K}_m(A,B)={\rm blkspan} \{B, AB, A^2 B , \dots A^{m-1} B \}.\]
The Lanczos iteration can then be compactly written as
\be\label{eq:LancRel}
A{\bm Q}_m={\bm Q}_m T_m + Q_{m+1} \Beta_{m+1}E_m^T, 
\ee
where $T_m$ is the symmetric positive definite block tridiagonal matrix 
\be\label{eq:T}
T_m=
\begin{pmatrix}
\Alpha_1 	& \Beta_2^T 	& {}		&{}			& {}& {}& {}\\
\Beta_2 	& \Alpha_2	& \Beta_3^T	&{}			& {}& {}& {}\\
{}			& \ddots 	& \ddots 	& \ddots 	& {}& {}& {}\\
{}			& {}		& \Beta_{i}& \Alpha_i & \Beta_{i+1}^T & {}& {}\\
{}			& {}		&{}			& \ddots 	& \ddots 	& \ddots& {}\\
{}			& {}		&{}			& {}	& \Beta_{m-1}& \Alpha_{m-1}& \Beta_m^T\\
{}			& {}		&{}			& {} 	& {} 	& \Beta_m& \Alpha_m\\
\end{pmatrix} ,
\ee
and $\Alpha_i, \Beta_i \in\mathbb{R}^{p \times p}$ are the block coefficients
in Algorithm~\ref{alg:blockLanc}. 

Using the Lanczos decomposition, $\blF(s)$ can be approximated as
\be\label{eq:blapprox}
 \blF(s)\approx \blF_m(s)= E_1^T\phi(T_m,s)E_1.
\ee
This approximation is known as a block Gauss quadrature rule.
In the following, we show that $\blF_m$ is monotonically increasing with 
$m$, i.e. $\blF_m< \blF_{m+1}<\blF(s)$ and thereby converges monotonically. Analogous results for the non-block case (i.e., $p=1$) for a larger class of functions were obtained in \cite{Druskin2008,lot2008,Frommer2009}. 
We also introduce a computable upper-bound matrix
$\blTF_m$ that monotonically decays as iterations increase, that is $\blTF_m>\blTF_{m+1}>\blF>\blF_m$. These two-sided bounds will then provide an error measure and allows us to reduce the approximation error using an averaging formula. The quantity $\blTF_m$ may be interpreted as an extension to $p>1$ of the approximation via a Gauss-Radau quadrature rule introduced in \cite{lot2008}.

\begin{remark}\label{rem2} 
	{ The simple Lanczos algorithm using only three-term recursions without re-orthogonalization 
 is known to be unstable due to computer round-offs. The instability is manifested by the loss of orthogonality of the Lanczos vectors and the appearance of spurious copies of the Lanczos eigenvalues. However, this instability only dampens the spectral adaptation of the approximation of actions of regular enough matrix functions and does not affect the converged approximation \cite{Greenbaum1989,druskin1995krylov,druskin1998using}. 
	Likewise, even though the spectral nodes and weights of the Gaussian quadratures computed via the simple Lanczos algorithm can be affected by round-off, the converged quadrature rule is not \cite{Knizhnerman1996TheSL,Golub1994}. Moreover, compared to the actions of matrix functions, the quadrature convergence is less affected by the loss of orthogonality of Lanczos vectors, because these vectors are not explicitly used in the computation of the Gaussian quadratures. We assume that the same reasoning applies to the block extensions as well to the Gauss-Radau quadrature, which we observe in our numerical experiments. Specifically, we assume that $T_m$ obtained from three-term block recursions is s.p.d. and (thereby) that its block Cholesky factorization exists. } 
	
\end{remark}

\section{Gauss quadrature and the matrix S-fraction} \label{sec:GaussSfrac}

We represent the block tridiagonal Lanczos matrix $T_m$ (see equation \eqref{eq:T}) in a 
block $LDL^T$ decomposition.

\begin{lemma}\label{lemma:T_LDLt}
	Let $T_m$ be as defined in equation~(\ref{eq:T}). Then
	there exist 
$\hKappa_{i}\in\RR^{p\times p}$, $\hKappa_1=I_p$ and $\gam_{i}\in\RR^{p\times p}$ all full rank such that
\begin{equation}\label{eq:defT}
T_{m} :=(\widehat{\boldsymbol K}_{m}^{-1})^T {J}_{m} \boldsymbol\Gamma^{-1}_m {J}_{m}^T \widehat{\boldsymbol K}_{m}^{-1}
\end{equation} 
where
\[
{J}_{m}^T = 
\begin{bmatrix}
I_p & -I_p & ~ & ~ \\
~ & \ddots& \ddots & ~ \\
 ~& ~& \ddots & -I_p \\
 ~ & ~&~& I_p 
\end{bmatrix}\in\RR^{pm\times pm},
	\begin{array}{ll}
\widehat{\boldsymbol K}_{m}&={\rm blkdiag}(\hKappa_{1},\dots,\hKappa_{m})\\
{\bm \Gamma}_m&={\rm blkdiag}(\gam_{1},\dots,\gam_{m}) ,
\end{array}
\]
and
$\Alpha_1=(\hKappa_{1}^{-1})^T\gam_1^{-1}\hKappa_{1}^{-1}=\gam_1^{-1}$, 
$\Alpha_i=(\hKappa_{i}^{-1})^T(\gam_{i-1}^{-1}+\gam_{i}^{-1})\hKappa_{i}^{-1}$ and 
$\Beta_i= -(\hKappa_{i}^{-1})^T \gam_{i-1}^{-1} \hKappa_{i-1}^{-1}$ for $i=2,\ldots,m$. 
\end{lemma}
	\begin{proof}
The block $LDL^T$ decomposition for positive definite, tridiagonal matrices consists of a block bidiagonal $L$ with identity as its diagonal block and the block diagonal s.p.d. $D$. The parametrization in equation \eqref{eq:defT} is a general parameterization of such a $LDL^T$ decomposition with $L=(\widehat{\boldsymbol K}_{m}^{-1})^T {J}_{m} (\widehat{\boldsymbol K}_{m})^T$ and $D=(\widehat{\boldsymbol K}_{m}^{-1})^T \boldsymbol\Gamma^{-1}_m (\widehat{\boldsymbol K}_{m}^{-1})$.

The $p\times p$ matrices must be of full rank, since they are Cholesky factors of the 
s.p.d. Lanczos matrix \eqref{eq:defT}, which is of full rank. Further by the same reasoning, the (block)-diagonal factor ${\bm \Gamma}_m$ is s.p.d. since $D$ is s.p.d.\,.
\end{proof}

 We can extract the matrices $\gam_i$'s and $\hKappa_{i}$'s directly during the block Lanczos recurrences using the coefficients $\Alpha_i$'s and $\Beta_i>0$'s. This is reported in Algorithm~\ref{alg:ExtractGam}, which in a different parameterization was derived in Appendix~C of~\cite{ZaslavskySfraction}.
From their construction, we obtain that $\hKappa_i$, $\gam_{i}$ are full rank, as long as no deflation occurs
in the block Lanczos recurrence, as $\Beta_i^T\Beta_i>0$ for $i=2,\ldots,m$, i.e. nonsingularity of all coefficients $\Alpha_i,\Beta_i$ ensures that 
Algorithm~\ref{alg:GammaExtraction} runs without breakdowns. This algorithm is also related to the matrix continued fraction construction via block Lanczos algorithm used for computation of quantum scattering problems~\cite{Jones1987TheRM}.
 
 {\begin{center}
\begin{minipage}{.55\linewidth}
 \begin{algorithm}[H]
 	\caption{Extraction algorithm $\gam/\hKappa$:\\ (Block $LDL^T$ Cholesky factorization of block tridiagonal $T_m$)}\label{alg:GammaExtraction}
 	\begin{algorithmic}
 	\normalsize
 		\State Given $\Alpha_i,\Beta_i$ and $\hKappa_1=I_p$
 		\State $\gam_1^{-1}=\hKappa_{1}^T \Alpha_1 \hKappa_{1}$ 		
		 \For{$i= 2,\dots, m$} 
 		\State $\hKappa_{i}^{-1}\quad	={ -} \gam_{i-1}(\hKappa_{i-1})^T\Beta_i^T \quad\, \phantom{-}(*)$
 		\State $\gam_i^{-1}\quad=\phantom{-} (\hKappa_{i})^T \Alpha_i \hKappa_{i}- \gam_{i-1}^{-1}\quad \,(\dagger)$
 		\EndFor

 	\end{algorithmic}
 	\label{alg:ExtractGam}
 \end{algorithm}
 \end{minipage}
 \end{center}
 }
 \vskip 0.1in

To write $\blF_m$ as a matricial continued-S-fractions we first convert $T_m$ to pencil form using 
the factors introduced in \eqref{eq:defT}. To this end, we introduce the matrices $\hgam_j$
\be\label{eq:decompose}
\hgam_j =\hKappa_j^T \hKappa_j, \quad j=1,\ldots,m.
\ee
The matrices $\gam_j$ and $\hgam_j$ are known as the Stieltjes parameters \cite{Duykarev} and they are both s.p.d. if block Lanczos runs without breakdown.

Let $ Z_m:= {J}_m \bm\Gamma_m^{-1} {J}_{m}^T$ and $\bm{\widehat\Gamma}_m={\rm blkdiag}(\hgam_{1},\dots,\hgam_{m})$.
Then, due to the initial condition $\hgam_1=I_p$, the function $\blF_m$ can be rewritten using the pencil
form $(Z_m,\bm{\widehat\Gamma}_m)$, that is
\be\label{eq:Gauss}
\blF_m(s)=E_1^T(T_m+sI)^{-1}E_1= E_1^T( Z_m+s \widehat{\bm \Gamma}_m)^{-1}E_1.
\ee

Next, we show that $\blF_m$ is a matricial Stieltjes continued fraction (S-fraction) that can be written as
\be\label{eq:S-fraction1}
	\blF_m(s)= \cfrac{1}{s\hgam_1+ \cfrac{1}{\gam_1 + \cfrac{1}{s\hgam_2 +\cfrac{1}{ \ddots \cfrac{1}{s \hgam_m + \cfrac{1}{ \gam_m}} }}}} ,
\ee
where the basic building block of these S-fractions is given by the recursion
\be\label{eq:DefCFblock}
	\blC_i(s)=\cfrac{1}{s\hgam_i + \cfrac{1}{\gam_i+\blC_{i+1}(s)}} ,
	\qquad {\rm Re}\{s\}>0.
\ee

\begin{proposition}\label{prop:Pencil2Stieltjes}
Let the recursion $\blC_i(s)$, $i=m-1, \ldots, 1$ as in (\ref{eq:DefCFblock}) be given.
Then for $s\in \RR_+$ and $\blC_{m+1}(s) = 0$, 
	the function $\blF_m$ is a continued matrix S-fraction defined as
\be\label{eq:DefS}
	\blF_m(s):=\blC_1 . 
\ee
\end{proposition}

This result can be formally related to earlier works by Duykarev et al. \cite{Duykarev,dyukarev2012matrix} in matricial continued S-fraction.
The proof used here is related to the derivation by Ingerman (Section~2.4 \cite{IngermanDtN}).

\begin{proof}
We recall that 
$\blF_m(s)$ corresponds to the first $p\times p$
block of the solution ${ U}_m$ of the linear system
$$
( Z_m+s \widehat{\bm \Gamma}_m){ U}_m =E_1.
$$
Let ${ U}_m = [\blU_1; \ldots; \blU_m]$, with $\blU_i$ a $p\times p$ block.
We notice that $ U_m$ is full column rank.
Let us write down the relevant block equations of this system, that is
\begin{eqnarray}
(\frac 1 \gam_1+s\hat\gam_1) \blU_1 - \frac 1 \gam_1 \blU_2 &=&I_p \label{eqn:line1}\\
- \frac 1 {\gam_{i-1}} \blU_{i-1} +(\frac 1 {\gam_i} + 
\frac 1 {\gam_{i-1}} +s\hat\gam_i) \blU_i - \frac 1 {\gam_{i}} \blU_{i+1}
 &=&0, \quad
i=2,\ldots, m-1 \label{eqn:linei}\\
- \frac 1 {\gam_{m-1}} \blU_{m-1} +(\frac 1{\gam_{m}}+\frac 1 {\gam_{m-1}}+s\hat\gam_m) \blU_m &=& 0. \label{eqn:linem}
\end{eqnarray}

As a preliminary observation, we note that we can assume
that all $\blU_i$ are full rank.
Indeed, since the coefficient matrix is block tridiagonal and unreducible,
if there exists a nonzero vector $v$ such that
$\blU_{i}v=0$ for some $i$, the components of $\blU_1v$ can be obtained
by solving the $(i-1)\times (i-1)$ system with the top-left portion
of the coefficient matrix\footnote{This is related to having solved the associated linear system. Indeed, let ${\bm Q}_i$ be the block Lanczos 
basis, and $R_i=(A+sI) {\bm Q}_i {U}_i - B$ be the residual matrix associated with
the approximate solution $X_i={\bm Q}_i {U}_i$ to the linear system $(A+sI)X=B$.
Using the Lanczos relation \eqref{eq:LancRel}, it holds that $R_i={\bm Q}_i (T_i+sI) {U}_i - {\bm Q}_iE_1 +
Q_{i+1} {\bm \beta}_{i+1} E_i^T {U}_i = 
Q_{i+1} {\bm \beta}_{i+1} E_i^T {U}_i$, with ${U}_i=[\blU_1;\ldots; \blU_i]$, 
that is $R_i=Q_{i+1} {\bm \beta}_{i+1} \blU_i$.}.

For $1<i \le m$, let us formally define the operation
$$
\blC_i^{-1} \blU_i:= - \frac 1 {\gam_{i-1}} (\blU_i-\blU_{i-1}) .
$$
Note that this allows us to write
\begin{eqnarray}\label{eqn:C}
\blC_i^{-1}\blU_i = (\blC_i + \gam_{i-1})^{-1} \blU_{i-1}.
\end{eqnarray}
Rearranging terms for $i=m$ in the equation (\ref{eqn:linem}), it holds
$$
-\blC_m^{-1} \blU_m + s \hat\gam_m \blU_m + \frac 1 {\gam_m} \blU_m = 0, 
$$
so that since $\blU_m$ is non-singular, 
we obtain $\blC_m^{-1}= s \hat\gam_m + \frac 1 {\gam_m}$.
Reordering terms similarly in the $i$th equation (\ref{eqn:linei}), 
with $1<i < m$, we obtain
$$
-\blC_i^{-1}\blU_i + s \hat\gam_i \blU_i + \blC_{i+1}^{-1} \blU_{i+1} = 0.
$$
Using (\ref{eqn:C}) with index $i+1$ in place of $i$ and substituting, we
rewrite the equation above as
$$
-\blC_i^{-1}\blU_i + s \hat\gam_i \blU_i + (\blC_{i+1}+\gam_i)^{-1} \blU_{i} = 0,
$$
yielding the following recurrence for $\blC_i$ (for $\blU_i$ nonsingular)
$$
\blC_i = \frac 1 {s \hat\gam_i + \frac 1 {\blC_{i+1}+\gam_i}}.
$$
Rearranging also the first equation (\ref{eqn:line1}), we have
$$
s\hat\gam_1 \blU_1 - \frac 1 \gam_1(\blU_2 - \blU_1) = I_p,
$$
so that $s\hat\gam_1 \blU_1 + \blC_2^{-1} \blU_2 = I_p$.
Using (\ref{eqn:C}), we obtain
$s\hat\gam_1 \blU_1 + (\blC_2+\gam_1)^{-1} \blU_1 = I_p$,
from which we get the final relation,
$$
\blU_1 = \frac 1 {s\hat\gam_1+\frac 1 {\blC_2+\gam_1}} \equiv \blC_1.
$$
\end{proof}

\section{Block Gauss-Radau quadrature}\label{sec:GaussRadau}
Let us modify the lowest floor of \eqref{eq:S-fraction1} by replacing $\gam_m$ with $\frac{\gam_m}{\epsilon}$, $\epsilon >0$ and keeping the remaining Stieltjes parameters unchanged. In the limit of $\epsilon \to 0$ such a modification yields 
	\be\label{eq:S-fraction2}
	\blTF_m(s) =\underset{{\epsilon\to 0}}{\lim}
	\cfrac{1}{s\hgam_1+ \cfrac{1}{\gam_1 + \cfrac{1}{s\hgam_2 +\cfrac{1}{ \ddots \cfrac{1}{s \hgam_m + \cfrac{\epsilon}{ \gam_m}} }}}}
	= \cfrac{1}{s\hgam_1+ \cfrac{1}{\gam_1 + \cfrac{1}{s\hgam_2 +\cfrac{1}{ \ddots \cfrac{1}{s \hgam_m } }}}}.
	\ee
Equivalently, such a limiting transition can be written using the resolvent of a block tridiagonal matrix as
	\[\blTF_m(s)=E_1^T(\tT_m+sI)^{-1}E_1,\]
	where 
	\be\label{eq:tT}
	\tilde T_m=
	\begin{pmatrix}
		\Alpha_1 	& \Beta_2^T 	& {}		&{}			& {}& {}& {}\\
		\Beta_2 	& \Alpha_2	& \Beta_3^T	&{}			& {}& {}& {}\\
		{}			& \ddots 	& \ddots 	& \ddots 	& {}& {}& {}\\
		{}			& {}		& \Beta_{i}& \Alpha_i & \Beta_{i+1}^T & {}& {}\\
		{}			& {}		&{}			& \ddots 	& \ddots 	& \ddots& {}\\
		{}			& {}		&{}			& {}	& \Beta_{m-1}& \Alpha_{m-1}& \Beta_m^T\\
		{}			& {}		&{}			& {} 	& {} 	& \Beta_m& \tilde \Alpha_m\\
	\end{pmatrix} ,
	\ee
	and $\Alpha_i$, $i=1,\ldots, m-1$ and $\Beta_i$, $i=2,\ldots, m$ are the same block coefficients as
	in Algorithm~\ref{alg:blockLanc},	
	with 
	\be 
	\label{eq:alpham0}
	\tilde \Alpha_m=\lim_{\epsilon \to 0}\Alpha_m(\epsilon)=
	 \lim_{\epsilon \to 0} (\hKappa_{m}^{-1})^T\left[ \gam_{m-1}^{-1}+ \left({\frac{\gam_m}{\epsilon}}\right)^{-1}\right]\hKappa_{m}^{-1}=
	(\hKappa_{m}^{-1})^T\gam_{m-1}^{-1}\hKappa_{m}^{-1},
	\ee using $\gam_m$, $\hKappa_{m}$ computed via Algorithm~\ref{alg:GammaExtraction}. The following proposition is fundamental for understanding the distinct properties of $\blF_m$ and $\blTF_m$ and the relation of $\blTF_m$ to the Gauss-Radau quadrature.

\vskip 0.1in
\begin{proposition}\label{lem03}
	The matrix function $\blF_m(s)$ is analytical and of full rank at $s=0$ as
	\be\label{eq:S(0)} 
	\blF_m(0)=\sum_{i=1}^m\gam_i. 
	\ee
	Further, $\blTF_m(s)$ has a pole of order one at $s=0$ with a 
	{ residue matrix of rank $p$}, specifically

	\be\label{eq:stS(0)}
	\lim_{s\to 0}(s\blTF_m)=\left(\sum_{i=1}^{m}
	\hgam_i\right)^{-1}. 
	\ee
\end{proposition}
\vskip 0.1in
\begin{proof}
	The basic building block of the considered S-fractions is given in \eqref{eq:DefCFblock},	which allows us to succinctly define the truncated matrix continued S-fraction~\eqref{eq:S-fraction2} recursively as
		\be\label{eq:DefTildeS}
		\blTF_m(s):=\blC_1, \quad {\rm for}\quad {\blC_{m+1}(s) = +\infty},
		\ee
		in the same fashion we defined $\blF_m(s)$ in equation \eqref{eq:S-fraction2} via \eqref{eq:DefS}.

	We will start with \eqref{eq:S(0)} and use the
	recursions \eqref{eq:DefCFblock} and \eqref{eq:DefS} to compute $\blF_m(0)$. First, we notice that $\blC_m(0)=\gam_m$. Then for $i\le m$ we get $\blC_{i-1}=\blC_{i}+\gam_{i-1}$ which proves
	\eqref{eq:S(0)} by induction.
	
	Likewise, \eqref{eq:stS(0)} can be obtained from the recursion \eqref{eq:DefCFblock} and \eqref{eq:DefTildeS}.
	We will start with computing $(\lim_{s\to 0} s\blC_{m})^{-1}=\hgam_{m}$ and then get the recursion
	\[(\lim_{s\to 0} {s\blC_{i-1}})^{-1}=\hgam_{i-1}+(\lim_{s\to 0} {s\blC_{i}})^{-1} \text{ for } i=m-2,\ldots,2.\] Then \eqref{eq:stS(0)} will be computed by substituting $\blC_1$ into the last equality of \eqref{eq:DefTildeS}. The inverse of the sum of s.p.d. matrices is of full rank, which concludes the proof.
\end{proof}

	The poles of $\blTF_m$ (as of any matrix S-fraction) are real non-positive and the pole at $s=0$ is of first order with residue matrix of rank $p$, which follows from the second result of Proposition~\ref{lem03}. Therefore, $\tT_m$ has eigenvalue zero with multiplicity $p$ and we arrive at the following corollary. 
	
	\begin{corollary}\label{prop:GR}
		The symmetric block tridiagonal matrix $\tT_m$ is semi-definite with $p$-dimensional null space.
	\end{corollary}
	
Because it is obtained from $T_m$ by modifying only the last diagonal block element $\tilde{\Alpha}$, $\tilde \blF_m$ can be viewed as a natural extension of the Gauss-Radau quadrature introduced in \cite{lot2008} for $p=1$, thus we will call it block Gauss-Radau quadrature. 

\begin{remark}
 Other choices of $\tilde \Alpha_m$ that lead to block Gauss-Radau rules are discussed in \cite{Lun18}, where $p$ distinct eigenvalues of the block tridiagonal matrix $\tilde T_m$ are prescribed by solving a block tridiagonal system. Here we consider discretizations of differential operators with a contentious spectrum where the lowest eigenvalue tends to zero such that we choose $p$ eigenvalues at zero.
\end{remark}

\begin{remark}\label{rem:Nullspace}
The theory presented in this manuscript was developed for s.p.d. matrices. In the case $A$ is a nonnegatvive definite matrix and the intersection of the column space of $B$ with the null space of $A$ is empty ($\mathrm{col}(B) \cap \mathrm{null}(A) = \{\varnothing\}$) we can still use block Lanczos algorithms to approximate the considered quadratic forms. In exact arithmetic $T_m$ would be positive definite for this case, a property that is lost in computer arithmetic. However, a spectral correction can be applied to the eigenvalues of $T_m$ that are not positive by expanding the spectral projectors introduced for the non-block case in \cite{1994RadioSc} to the block case. Essentially, eigenvalues below a spectral threshold are shifted to the threshold and block Lanczos is rerun on the spectrally corrected matrix to restore the block tridiagonal structure. In the numerical experiments presented in sections \ref{sec:EMdiff} and \ref{sec:graphL} the considered matrix $A$ is nonnegative definite, yet no spectral corrections were needed.
\end{remark}

 \section{Two-sided error bounds and convergence acceleration via averaging of Gauss and Gauss-Radau rules}\label{sec:ErrBound}

 \subsection{Monotonically convergent lower and upper bounds}
 In the absence of deflation, Gaussian quadrature $\blF_m(s)$, as all Krylov approximations eventually do, converges to $\blF(s)$. Gauss-Radau quadrature $\blTF_m(s)$ differs from $\blF_m(s)$ only by the absence of the lower floor in the S-fraction representation, thus it also converges to the same limit. We next prove that the two sequences give monotonic two-sided bounds and that their difference defines an upper bound.

 \vskip 0.1in
 \begin{theorem}\label{lem01}
 	For a given real positive $s$, the following holds,
 	
 	i) The matrix sequence $\{ \blF_m(s)\}_{m\geq 1}$ is increasing, i.e. 
 	$\blF_m(s)<\blF_{m+1}(s)$;
 	
 	ii) The matrix sequence $\{ \blTF_m(s)\}_{m\geq 1}$ is decreasing, 
 	i.e. $\blTF_m(s)>\blTF_{m+1}(s)$.
 	
 \end{theorem}
 
 \vskip 0.1in
 \begin{proof} { From the recursion 
 		\eqref {eq:DefCFblock} with final condition given by \eqref {eq:DefS} 
 		$\forall s>0$ and $i<m+1$ we get $\blC_i(s)>0$.
 		From \eqref{eq:DefCFblock} it follows that we get 
 		$\blF_{m-1}$ if we impose $\blC_{m}=0$. 
 		From \eqref{eq:DefCFblock} it also follows
 		that $\blC_i$ is monotonic with respect to $\blC_{i+1}$ ($\blC_{i+1}$ in the denominator of $\blC_{i}$'s denominator), so the replacement of the final condition in \eqref {eq:DefS} by $\blC_{m}=0$ decreases $\blC_{m-1}$, that consequentially decreases $\blC_{m-2}$, etc., and by induction we obtain $\blF_{m-1}<\blF_m$. } 
 	
 	For proving (ii), we recall that
 	one can get $\blTF_{m-1}$ by terminating the S-fraction with $\blC_{m}=+\infty$. 
 	Then with the same monotonicity argument as for $\blF_m$, it follows 
 	that such a substitution recursively increases $\blC_i$, $i=m-1,m-2,\ldots, 1 $.
 \end{proof}

Now we are ready to formulate a two-sided Gauss-Radau-type estimate.

\vskip 0.1in
\begin{proposition}\label{prop:03}
{ Let Algorithm~\ref{alg:blockLanc} produce 
	$\Beta_j^T\Beta_j>0$ for $j\le m$. Then
	\[
	0<\blF_{m-1}(s)<\blF_{m}(s)< \blF(s)< \tilde\blF_{m}(s)<\tilde \blF_{m-1}(s) \quad \forall s\in\RR_+.
	\]}
\end{proposition}

\begin{proof}
We apply Theorem~\ref{lem01} to $\blF_m$ and $\blTF_m$, yielding $0<\blF_{m-1}(s)<\blF_{m}(s)$ 
and $\blTF_{m}(s)<\blTF_{m-1}(s)$. 
To conclude the proof, we notice that for $s$ satisfying the 
condition of the proposition 
$\lim_{m\to\infty} \blF_m(s)=\blF(s)$, i.e., $\blF_m(s)<\blF(s)$. 
From $T_m- \tilde T_m\ge 0$ we obtain 
{
	\begin{eqnarray*}
		\tilde \blF_m(s)-\blF_m(s)&=&E_1^T(\tilde T_m+sI)^{-1}E_1-E_1^T (T_m+sI)^{-1}E_1\\
		&=&\int_{0}^{1}\frac{d}{dr}E_1^T(T_m+ r(\tilde T_m-T_m)+sI)^{-1}E_1 d r\\
		&=&-\int_{0}^{1}V(r)^T(\tilde T_m-T_m)V(r) dr \ge 0,
	\end{eqnarray*}
}
with $V(r)=[T_m+ r(\tilde T_m-T_m)+sI]^{-1}E_1$, that concludes the proof.
\end{proof}
\vskip 0.1in
The ordering above provides us with a computable error estimate of the approximation of $\blF$, that is
\begin{equation}\label{eq:ErrorBound}
\|\blF- \blF_m \| < \|\blTF_m - \blF_m \|.
\end{equation}

{The two-sided bounds of Proposition~\ref{prop:03} indicate the possibility of using their averaging to compute an error correction. 
Such an approach was considered in the works of Reichel and his 
co-authors for Gauss-Radau quadrature with $p=1$ \cite{lot2008} and for $p>1$ for anti-Gauss quadratures \cite{lot2013}. Here we identify an important class of PDEs operators for which the averaging is particularly efficient. 
%

\subsection{Acceleration via averaged Gauss-Radau and Gauss rules}\label{sec:Averging}
{ We propose two averaging formulas for the developed block Gauss-Radau and block Gauss rules to accelerate convergence. We motivate and analyze these averaging formulas in Appendix~\ref{sec:PostQuadErr}.

The first averaging formula, which we refer to as ``Averaging 1" in the numerical experiments, is the arithmetic average of the Gauss and Gauss-Radau approximation
\begin{equation}\label{eq:Extrapol1}
\widehat \blF_m(s):=\frac 1 2 \left(\blF_m(s)+\tilde \blF_{m+1}(s)\right).
\end{equation}
We stress that the computation of 
$\tilde \blF_{m+1}(s)$ requires the same $m$ steps of the block 
Lanczos algorithm as $ \blF_{m}(s)$. 


In Appendix~\ref{sec:PostQuadErr} (for $p=1$) we show that for the cases with linear convergence (corresponding to the problems arising from discretzation of operators with continuous spectrum, e.g., PDEs in unbounded domains) such an averaging leads to approximate error reduction by a factor equal to the linear convergence rate of Gauss or Gauss-Radau quadratures, so the proposed averaging will be beneficial for problems with slow linear convergence.


The second averaging formula, which we refer to as ``Averaging 2" in the numerical experiments is the geometric mean
	\begin{equation}\label{eq:Extrapol2}
	\widecheck \blF_m(s):= \exp\left[\frac 1 2 [\log(\overline{\blF_m}(s))+\log(\widehat \blF_{m}(s))]\right] 
	\end{equation}
	of the arithmetic mean $\widehat \blF_m(s)$ and harmonic mean 
	\begin{equation}
	\overline{\blF_m}^{-1}:=\frac 1 2 \left [(\blF_m(s))^{-1}+(\tilde \blF_{m+1}(s))^{-1}\right].
	\end{equation}
This average is invariant if we substitute $\blF^{-1}$ for $\blF$ and reverse their order, i.e. the averaging formula for $\blF$ and $\blF^{-1}$ are the same. This is a significant property of relevance for matrices $A$ stemming from discretizations of 2D PDEs as explained in Appendix~\ref{sec:PostQuadErr}. The simple averaging formula~(\ref{eq:Extrapol1}) lacks this property.

}

\section{Numerical Examples}\label{sec:NumEx}

In this section, we consider the application of the developed error bounds and show that the averaging formulas accelerate convergence for MIMO transfer function computations for PDE operators with continuous or very dense spectrum when Krylov subspace approximations lose their spectral adaptivity. This phenomenon can only be observed for large-scale problems, thus dimensions of our experiments are of order $10^5-10^6$. Though the theory in this manuscript is limited to the resolvent with $s\in \mathbb{R}^+$, we stress that the error bound, monotonicity and error correction via averaging seem to hold for imaginary $s$ and other functions such as $\phi(A)=\exp(-A t)$. We also explore possible accelerations stemming from adding random vectors to the initial block.
All errors reported are measured in the 2-norm.

\subsection{Transfer functions of PDE operators on unbounded domains}\label{sec:PDE}
The computation of transfer functions of self-adjoint PDE operators on unbounded domains is the main target of the proposed approach. For accurate PDE discretization of such domains, we adopted the optimal grid approach~(\cite{idkGrids}) allowing spectral convergence in the exterior part of the computational domain, e.g., see \cite{druskin2016near} with a more detailed explanation of optimal grid implementation. A trade-off of the spectral approximation of the exterior problem is that $A$'s spectrum approximates the continuum spectrum of the exact problem well, and thus the Krylov subspace approximations lose their adaptability. The proposed averaging attempts to alleviate this problem by adapting to continuous spectra.

\subsubsection{2D Diffusion}\label{sec:subsubsection2D}

We discretize the operator
\[
\sigma({\bf x})^{-\frac{1}{2}}\Delta\sigma({\bf x})^{-\frac{1}{2}}
\]
on an unbounded domain using second-order finite differences on a $300\times 300$ grid. In the grid center where the variations of $\sigma(\bx)$ are supported, we use a grid size of $\delta_x=1$. { This formulation can be associated with both heat transfer and the scalar diffusion of electromagnetic fields. To approximate this, we consider a Dirichlet problem on a bounded domain with $N_{opt}=10$ exponentially increasing grid steps in the exterior and a uniform grid in the interior part. Utilizing the optimal geometric factor of $\exp{(\pi/\sqrt{10})}$ as indicated by \cite{idkGrids}, we achieve spectral accuracy in the exterior region of the domain, thereby effectively approximating the boundary condition at infinity (exponential decay).} The resulting grid and function $\sigma(\bx)$ are shown in Figure~\ref{fig:Heat2config}. Only the first two steps of the geometrically increasing grid are shown in the figure since they increase rapidly. In this case, $B$ is a vector of discrete delta functions with entries at grid nodes corresponding to the transducer locations shown in the same figure as {\it Trx}. 

The convergence for $s=10^{-3}$ is shown in Figure~\ref{fig:HeatReal} alongside the convergence curve for purely imaginary shifts $s=\imath\cdot10^{-3}$ in Figure~\ref{fig:HeatImag}. For imaginary shifts $s$, block Lanczos converges slower.  The monotonicity result from Theorem~\ref{lem01} and the error bound from equation~(\ref{eq:ErrorBound}) are only proven for $s\in\mathbb{R}^+$. Despite this the convergence for purely imaginary $s$ is monotone in this example, the upper bound holds and the averaged quadrature rules reduce the approximation error by an order of magnitude. For the chosen shifts both averaging formulas perform equally well.

In Figure~\ref{fig:HeatReal} we compare the introduced methods with the averaged anti-Gauss rules developed in \cite{lot2013}. An anti-Gauss rule can be obtained by modifying the last off block diagonal entry of $T_{m+1}$, which we call $\Beta_{m+1}$, by multiplication with the factor $\sqrt{2}$. Let this matrix be called $T_{m+1}^{\cal H}$. There is no guarantee that this matrix is positive definite, and in about 38\% of the iterations for the 2D diffusion problem it is not. In this experiment the averaged block anti-Gauss rule
\be
\widehat{ \cal F}^{\cal H}_m = \frac{1}{2}\left( {\cal F}_m + E_1^T (T_{m+1}^{\cal H} +sI)^{-1}E_1 \right)
\ee
has the same error order as the block Gauss rule but is less stable due the presence of negative eigenvalues.

In Figure~\ref{fig:HeatExp} the convergence for the matrix exponential $\phi(A)=\exp{(-At)}$ for $t=10^{5}$ is shown. The convergence is superlinear, the error bound becomes tight with increased iterations and the averaged quadrature rule from formula~\eqref{eq:Extrapol2} performs better for low iteration counts. Numerical results for accelerating convergence further through subspace enrichment by adding random starting vectors to $B$ in $\mathcal{K}_m(A,B)$ are presented in Appendix~\ref{sec:Enrichment}.


\begin{figure}[h!]
	\centering
	\includegraphics[width = 0.65 \linewidth]{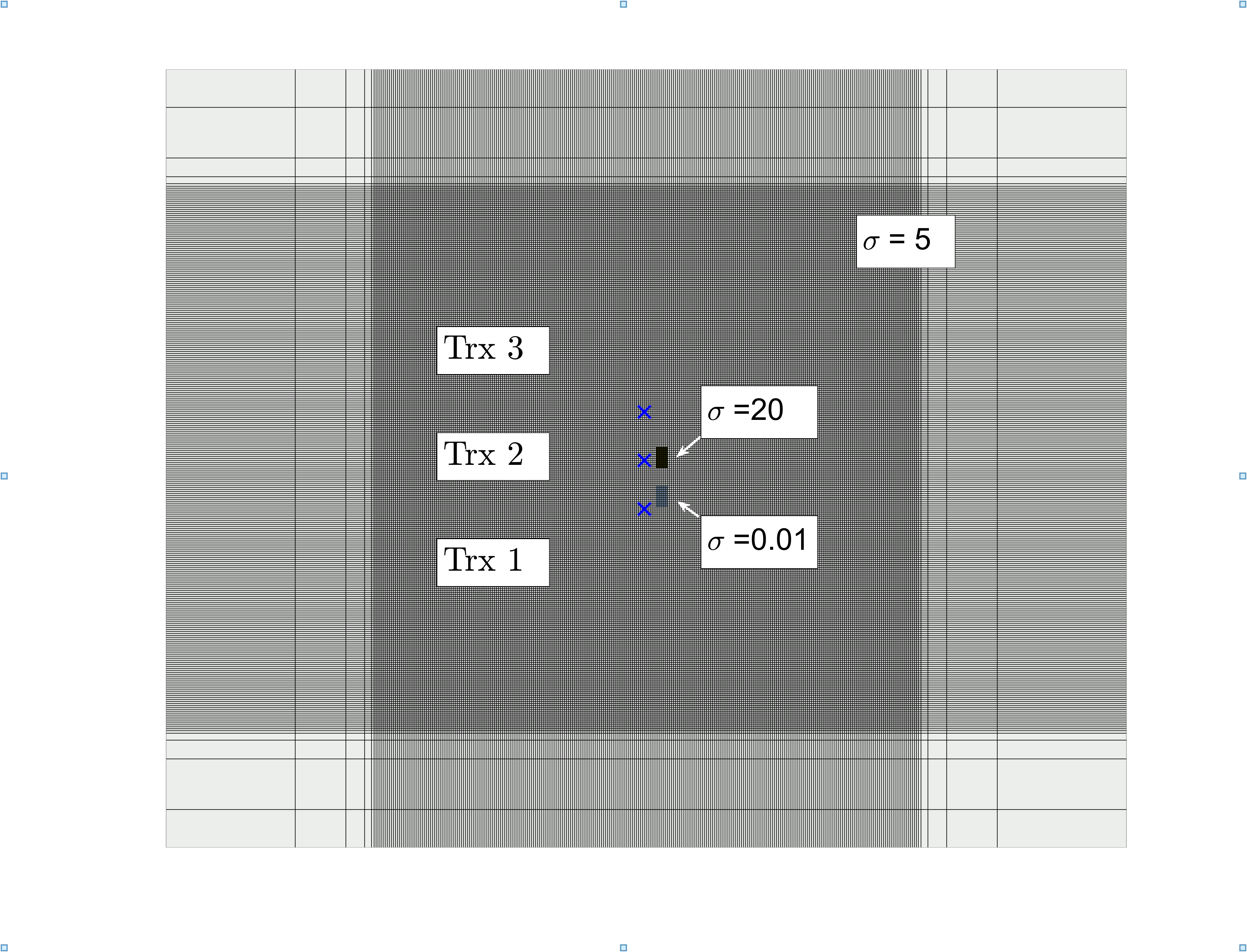}
	\caption{Grid, head conductivity $\sigma(\bx)$ and transducer locations of the heat diffusion textcase.}\label{fig:Heat2config}
\end{figure}

\begin{figure}[h!]
	\centering	
	\begin{subfigure}[b]{.48\linewidth}
		\centering
		\includegraphics[width = \linewidth]{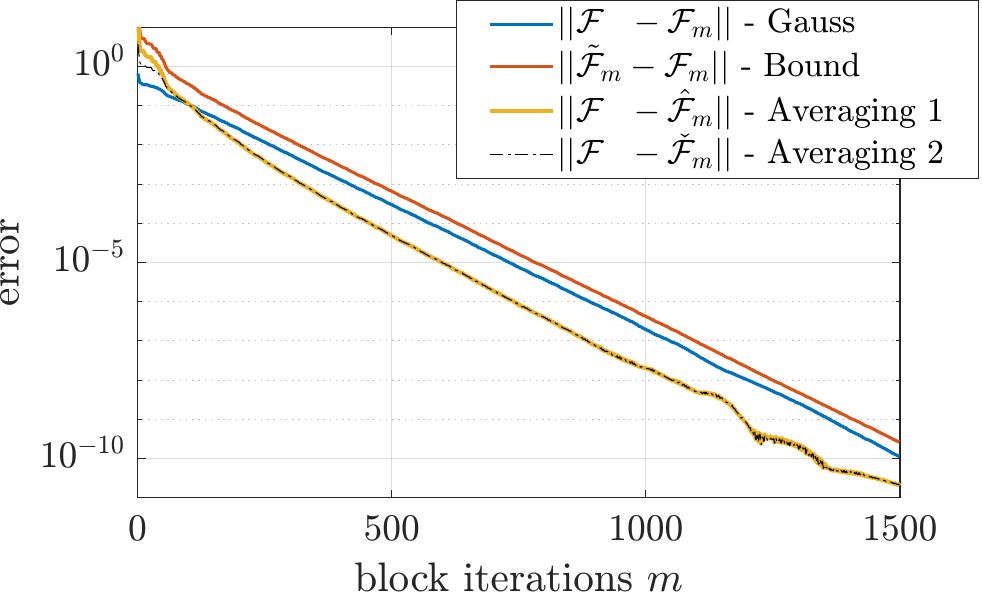}
		\caption{Error norms for {\bf imaginary} $s=0.001i$.}\label{fig:HeatImag}
	\end{subfigure}
	\begin{subfigure}[b]{.48\linewidth}
		\centering
		\includegraphics[width = \linewidth]{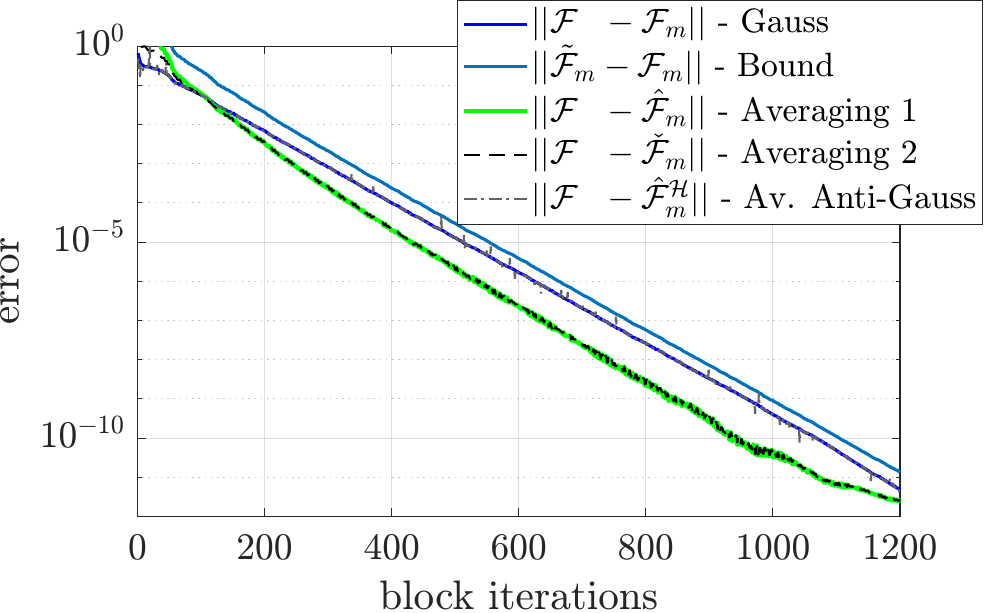}
		\caption{Error norms for {\bf real} shift $s=0.001$}\label{fig:HeatReal}
	\end{subfigure}%
	
	\begin{subfigure}[b]{.48\linewidth}
		\centering
		\includegraphics[width = \linewidth]{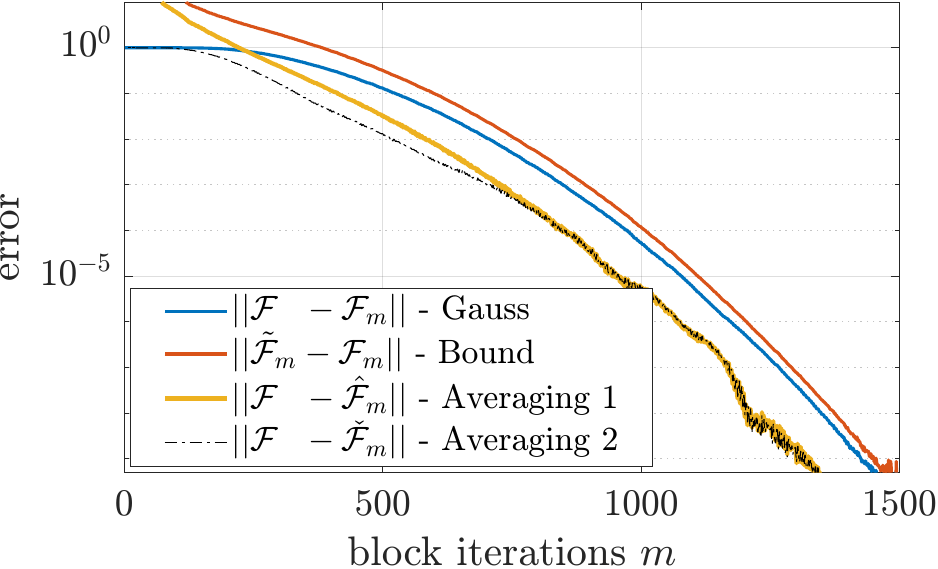}
		\caption{Error norms for the exponential function $\phi(A)=\exp{(-At)}$ for $t=10000$.}\label{fig:HeatExp}
	\end{subfigure}

	\caption{Convergence and quadrature rule averaging for the 2D heat diffusion problem. }
\end{figure}

\subsubsection{3D Electromagnetic diffusion}\label{sec:EMdiff} In this example we solve the 3D Maxwell equations in the quasi-stationary (diffusive) approximation. Krylov subspace approximation of action of matrix functions for the solution of that problem was introduced in \cite{druskin1988spectral} and remains the method of choice for many important geophysical applications. The diffusive Maxwell equations in $\RR^{3}$ can be transformed to the form
\be\label{eq:max}
(\nabla \times \nabla \times +\sigma(\bx) \mu_0 s)\mathcal{E}_{(r)}(\bx,s) = -  s\mathcal{J}^{\rm ext}_{(r)}(\bx,s), \quad \bx\in \Omega,
\ee
where $\sigma$ is the electrical conductivity, $\mu_0$ the vaccum permeability, $\mathcal{J}^{\rm ext}_{(r)}$ is an external current density indexed by transmitter index $(r)$ causing the electrical field strength $\mathcal{E}_{(r)}(\bx,s)$. For $s\notin\RR_-$ the solution of $\eqref{eq:max}$ converges to $0$ at infinity, i.e. \be\label{eq:inftymax}\lim_{\bx\to\infty} \|E(\bx)\|=0. \ee 
We approximate \eqref{eq:max}-\eqref{eq:inftymax} by the conservative Yee grid similarly to \cite{druskin1988spectral}.
 Following the same framework as in the diffusion problem in $\RR^2$ described earlier, we truncate the computational domain with Dirichlet boundary conditions for the tangential components of $\mathcal{E}_{(r)}$ on the domain boundary with a optimally geometrically progressing grid, which gives spectral convergence of the condition~\eqref{eq:inftymax} at infinity. In our particular realization, we used $N_{opt}=6$ gridsteps geometrically increasing with geometric scaling factor $\exp{(\pi/\sqrt{6})}$ according to \cite{idkGrids}. The grid is $N_x=80$, $N_y=100$ and $N_z=120$ gridsteps large leading to a matrix $A$ of size $N=2\, 821\, 100$. The configuration has a homogeneous conductivity of $\sigma=1$ in the entire domain except for two inclusions shown in Figure~\ref{fig:4_config} where $\sigma=0.01$. As current densities, we use 6 magnetic dipoles (approximated by electric loops), with 3 dipoles located above the inclusion and 3 dipoles located in the middle of the inclusions as highlighted by the arrows in Figure~\ref{fig:4_config}. Such configuration corresponds to the tri-axial borehole tools used in electromagnetic hydrocarbon exploration \cite{saputra2024adaptive} and the matrix $B$ thus has 6 columns. 

The operator and its discretization are nonnegative definite; however, the vectors $B$ corresponding to divergence-free magnetic sources are intrinsically orthogonal to the operator's null space, consisting of potential solutions. Thus, it satisfies the conditions in Remark~\ref{rem:Nullspace}.

In Figure~\ref{fig:EMSweeps} the convergence is shown for various real values of $s$ after $m=400$ block iterations. The error bound is tight across all values for $s$ 
and there is no difference in averaging formulas. Further, in Figure~\ref{fig:EM_conv_Real} the convergence for this test case is shown for a real shift of $s=0.05$ and in Figure~\ref{fig:EM_conv_Imag} for an imaginary shift.

\begin{figure}[h!]
 \centering
 
 \begin{subfigure}[b]{.47\linewidth}
 \centering
 \includegraphics[width = 0.75\linewidth]{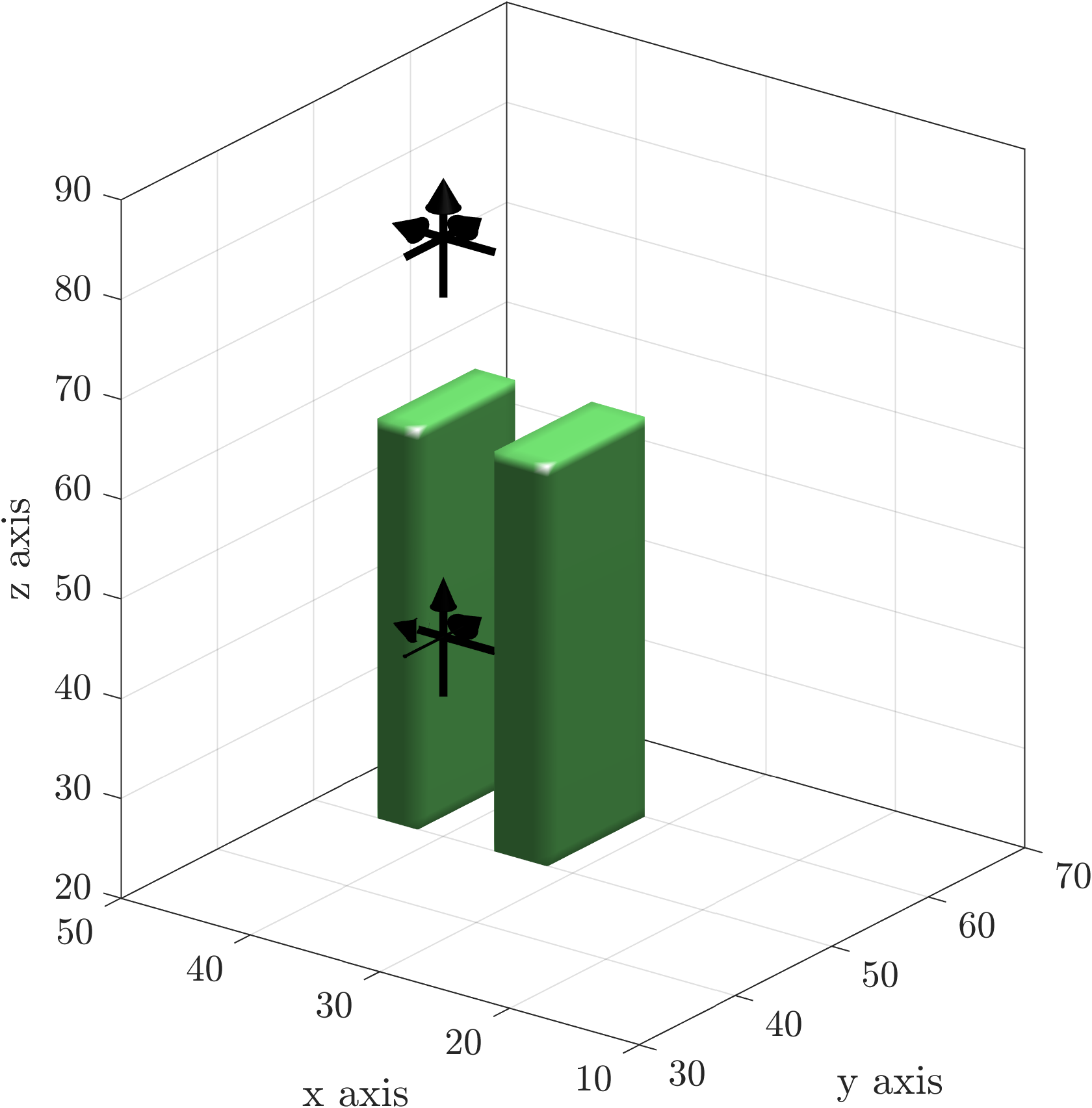}
 \caption{3D rendering of the simulated configuration with two inclusions.}
 \label{fig:4_config}
 \end{subfigure}%
 \hspace{1em}
 \begin{subfigure}[b]{.47\linewidth}
 \centering
 \includegraphics[width = \linewidth]{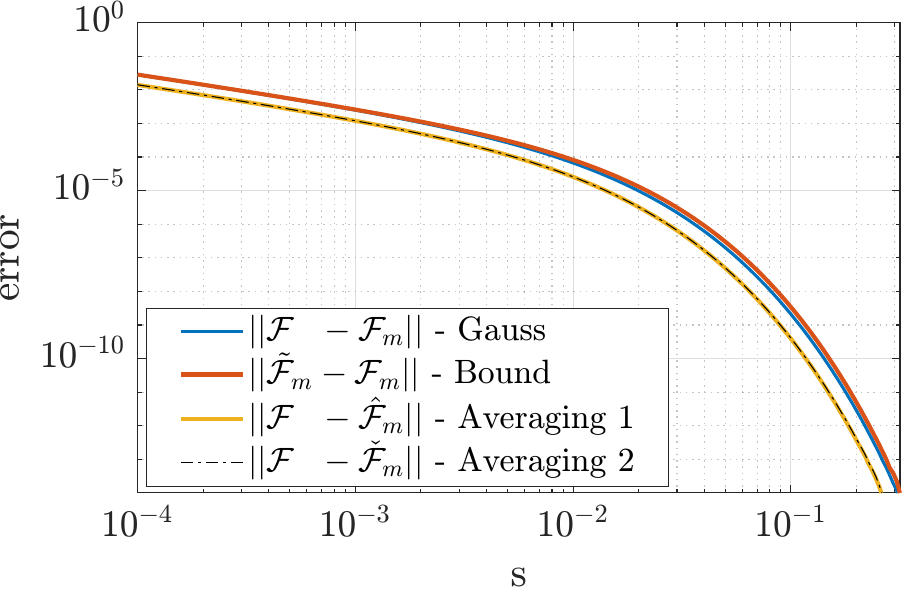}
 \caption{Error norms for various values of $s$ after $m=400$ iterations.}
 \label{fig:EMSweeps}
 \end{subfigure}
  \begin{subfigure}[b]{.47\linewidth}
 \centering
 \includegraphics[width = \linewidth]{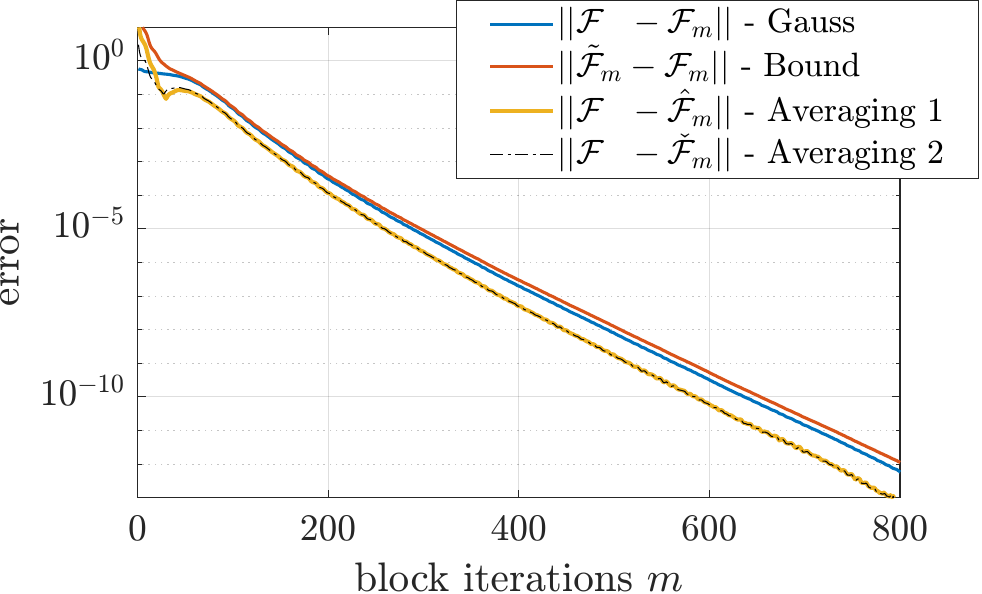}
 \caption{Error norms with real shift $s=0.05$. }
 \label{fig:EM_conv_Real}
 \end{subfigure}%
 \hspace{1em}
 \begin{subfigure}[b]{.47\linewidth}
 \centering
 \includegraphics[width = \linewidth]{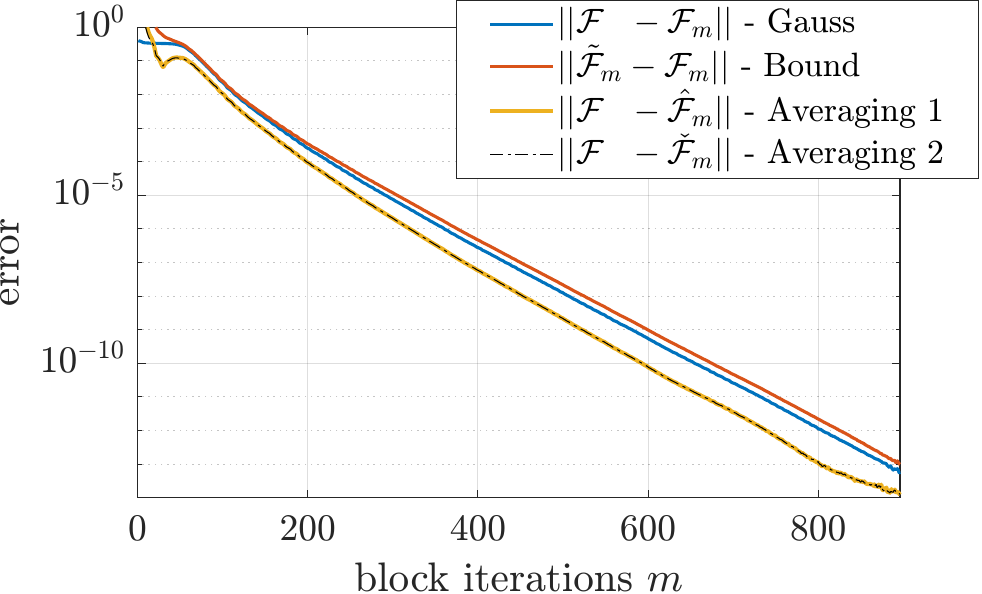}
 \caption{Error norms with {\rm imag.} shift $s=0.1\imath$.}
 \label{fig:EM_conv_Imag}
 \end{subfigure}%
\caption{Convergence of the block Lanczos method and averaging for the 3D electromagnetic case.}
\end{figure}

 \subsection{Graph Laplacian}\label{sec:graphL} { Efficient computation of SISO and MIMO transfer functions of large graph-Laplacians is important for efficient clustering of big data sets, computation of centrality measures and other applications \cite{lot2008,Benzi2020MatrixFI,druskin2022distance}. For our numerical experiments we} take the \texttt{web-BerkStan} graph from the SNAP (Standford network analysis project) \cite{snapnets} and turn it into an undirected graph by symmetrization. The original graph has 685230 nodes and 7600595 edges and we construct the normalized graph Laplacian $A$ as our test matrix. The nodes of the graph represent pages from berkeley/stanford.edu domains and directed edges (here converted to undirected edges) represent hyperlinks between them. The data were collected in 2002. 
 
As $B$ matrix we take $p=3$ (graph) delta functions chosen at random nodes\footnote{With random seed 1 we obtain delta functions at nodes 207168, 100562, and 63274.}. The graph Laplacian is nonnegative definite, where the null space of the graph Laplacian consists of constant vectors on every connected graph component. The columns of $B$ are graph delta functions in our experiments and thus not orthogonal to the null space; however, excluding the case of connected graph components of dimension 1 (single disconnected nodes), the intersection of the column space of $B$ with the null space of $A$ is empty, and as such, satisfies the condition of Remark~\ref{rem:Nullspace}.
 
 In Figure~\ref{fig:GraphLap} we evaluate the resolvent $\phi(A,s)=(A+s I)^{-1}$ for $s=10^{-4}$ showing that the error bound is tight and the averaging formulas give similar results. After $m=50$ block iterations the error for various values of $s$ is displayed in Figure \ref{fig:GraphLap2}.

 \begin{figure}[h!b]
 	\centering
 	\begin{subfigure}[b]{.48\linewidth}
 		\centering
 		\includegraphics[width = \linewidth]{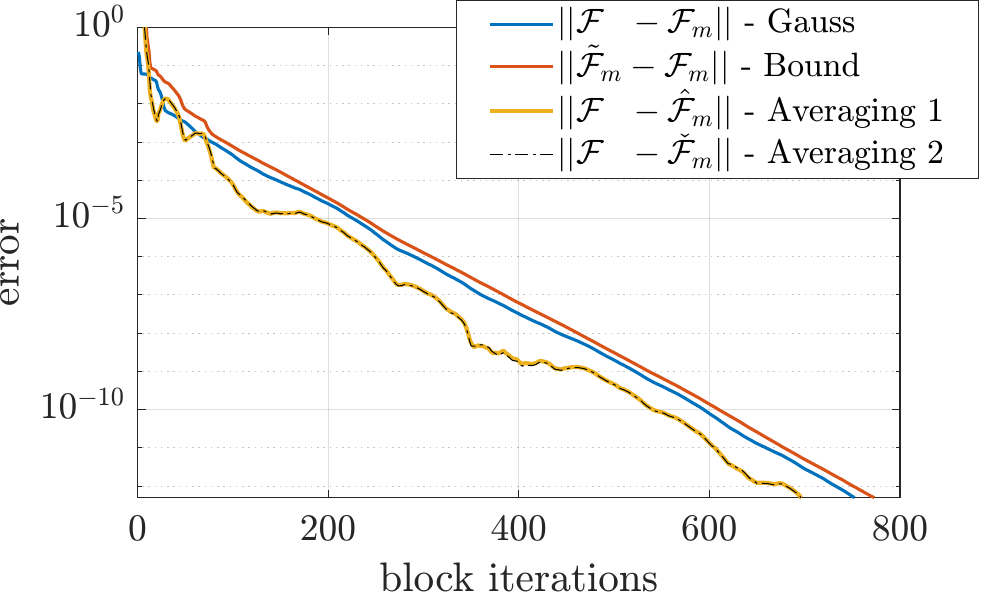}
 		\caption{Error norms with shift $s=10^{-4}$. \phantom{just force a line break}}
 		\label{fig:GraphLap}
 	\end{subfigure}~
 	\begin{subfigure}[b]{.48\linewidth}
 		\centering
 		\includegraphics[width = 0.93\linewidth]{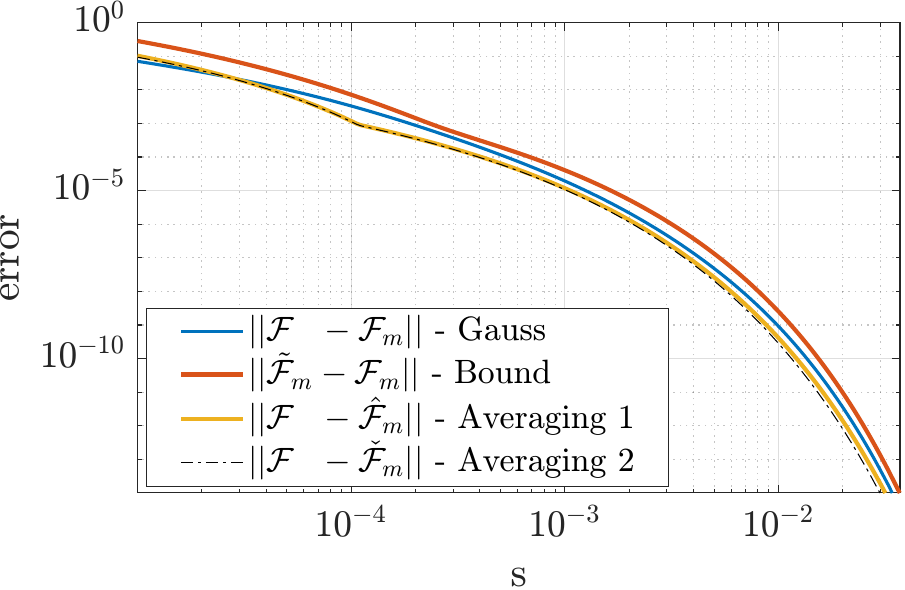}
 		\caption{Error norms after $m=50$ iterations for various values of $s$.}
 		\label{fig:GraphLap2}
 	\end{subfigure}%
 	\caption{Convergence, error bound and averaged quadrature rules for the normalized graph Laplacian test case}
 \end{figure}
 
 { In contrast to PDE operators on unbounded domains, data-related graphs may exhibit spectral distributions with non-vanishing gaps even at very large dimensions. Consequently, they cannot be regarded as approximations of the problems discussed in Proposition~\ref{prop:corr}. Nevertheless, as illustrated in Figure~\ref{fig:GraphLap}, quadrature approximations for such graphs can demonstrate linear convergence, in which case they satisfy tight block Gauss-Radau bounds and their averaging will be efficient. }

\section{Conclusions}
In this paper, we derived a block extension of Gauss-Radau quadrature for the computation of MIMO transfer functions by leveraging the connection between the block Lanczos algorithm and matrix Stieltjes continued fractions. We proved that the block Gauss and Gauss-Radau quadrature rules provide monotonically convergent two-sided bounds for the resolvent, important for MIMO transfer functions of LTI dynamical systems with real Laplace frequencies. For problems involving the approximation of operators with a continuous spectrum, we also prove (in the SISO case) that these bounds are tight, thereby justifying  the averaging of Gauss and Gauss-Radau approximations.

Our numerical experiments with large-scale graph-Laplacian, 2D diffusion, and 3D Maxwell systems confirmed the tightness of the derived bounds and the efficiency of the averaging formula, not only for the cases analyzed but also for transfer functions of LTI systems with imaginary Laplace frequency and time domain responses. Additionally, we explored the potential of using random enrichment of the initial Lanczos block within the developed algorithm framework. These findings highlight the robustness and broad applicability of our approach targeted for large-scale problems.

\section{Data Availability}

The datasets generated during and/or analyzed during the current study are available from the corresponding author upon reasonable request.

\begin{acknowledgements}
Part of this work was performed while Valeria Simoncini (VS) was in residence at the Institute for Computational and Experimental Research in Mathematics in Providence, RI, during the Numerical PDEs: Analysis, Algorithms, and Data Challenges semester program.

The work of VS was also funded by the European Union - NextGenerationEU under the National Recovery and Resilience Plan (PNRR) - Mission 4 Education and research - Component 2 From research to business - Investment 1.1 Notice Prin 2022 - DD N. 104 of 2/2/2022, entitled ``Low-rank Structures and Numerical Methods in Matrix and Tensor Computations and their Application'', code 20227PCCKZ – CUP J53D23003620006. VS is a member of the INdAM Research Group GNCS; its continuous support is gladly acknowledged.

The work of Vladimir Druskin was partially supported by AFOSR grants FA 955020-1-0079, FA9550-23-1-0220, and NSF grant DMS-2110773.

The authors like to thank Leonid Knizhnerman for helpful discussions related to the material in section~\ref{sec:PostQuadErr} and Michele Benzi for discussions related to the spectra occurring in large graphs.

The authors would like to thank the anonymous reviewers for their comments and for pointing out relevant literature missed in a previous version of this manuscript.
\end{acknowledgements}

%
\section*{Conflict of interest}
The authors declare that they have no conflict of interest.

\bibliographystyle{spmpsci} 
\bibliography{bib} 

\begin{thebibliography}{10}
\providecommand{\url}[1]{{#1}}
\providecommand{\urlprefix}{URL }
\expandafter\ifx\csname urlstyle\endcsname\relax
  \providecommand{\doi}[1]{DOI~\discretionary{}{}{}#1}\else
  \providecommand{\doi}{DOI~\discretionary{}{}{}\begingroup
  \urlstyle{rm}\Url}\fi

\bibitem{amsel2023nearoptimal}
Amsel, N., Chen, T., Greenbaum, A., Musco, C., Musco, C.: Near-optimal
  approximation of matrix functions by the {Lanczos} method.
\newblock In: Proceedings of the Thirty-Eighth Annual Conference on Neural
  Information Processing (2024)

\bibitem{Baker_Graves-Morris_1996}
Baker, G.A., Graves-Morris, P.: Padé Approximants, 2 edn.
\newblock Encyclopedia of Mathematics and its Applications. Cambridge
  University Press (1996)

\bibitem{GaussRadauBoundNetwork}
Benzi, M., Boito, P.: Quadrature rule-based bounds for functions of adjacency
  matrices.
\newblock Linear Algebra and its Applications \textbf{433}(3), 637--652 (2010)

\bibitem{Benzi2020MatrixFI}
Benzi, M., Boito, P.: Matrix functions in network analysis.
\newblock GAMM‐Mitteilungen \textbf{43} (2020)

\bibitem{BenziMatrixF}
Benzi, M., Golub, G.H.: Bounds for the entries of matrix functions with
  applications to preconditioning.
\newblock BIT Numerical Mathematics \textbf{39}, 417 --438 (1999)

\bibitem{PhysRevLett.65.325}
Berryman, J.G., Kohn, R.V.: Variational constraints for electrical-impedance
  tomography.
\newblock Phys. Rev. Lett. \textbf{65}, 325--328 (1990)

\bibitem{Cullum.Willoughby.85}
Cullum, J., Willoughby, R.A.: Lanczos algorithms for large symmetric eigenvalue
  computations.
\newblock Birkha\"{u}ser, Basel (1985).
\newblock Vol.1, Theory, Vol.2. Program

\bibitem{Druskin2008}
Druskin, V.: On monotonicity of the {Lanczos} approximation to the matrix
  exponential.
\newblock Linear Algebra and its Applications \textbf{429}(7), 1679--1683
  (2008)

\bibitem{druskin1998using}
Druskin, V., Greenbaum, A., Knizhnerman, L.: Using nonorthogonal {Lanczos}
  vectors in the computation of matrix functions.
\newblock SIAM Journal on Scientific Computing \textbf{19}(1), 38--54 (1998)

\bibitem{druskin2016near}
Druskin, V., G\"uttel, S., Knizhnerman, L.: Near-optimal perfectly matched
  layers for indefinite {Helmholtz} problems.
\newblock SIAM Review \textbf{58}(1), 90--116 (2016)

\bibitem{druskin1988spectral}
Druskin, V., Knizhnerman, L.: A spectral semi-discrete method for numerical
  solution of {3D} non-stationary problems in electrical prospecting.
\newblock Phys. Sol. Earth \textbf{24}, 641--648 (1988)

\bibitem{1994RadioSc}
Druskin, V., Knizhnerman, L.: Spectral approach to solving three-dimensional
  {Maxwell}'s diffusion equations in the time and frequency domains.
\newblock Radio Science \textbf{29}(4), 937--953 (1994)

\bibitem{druskin1995krylov}
Druskin, V., Knizhnerman, L.: Krylov subspace approximation of eigenpairs and
  matrix functions in exact and computer arithmetic.
\newblock Numerical linear algebra with applications \textbf{2}(3), 205--217
  (1995)

\bibitem{ZaslavskySfraction}
Druskin, V., Mamonov, A.V., Zaslavsky, M.: Multiscale {S}-fraction
  reduced-order models for massive wavefield simulations.
\newblock Multiscale Modeling \& Simulation \textbf{15}(1), 445--475 (2017)

\bibitem{druskin2022distance}
Druskin, V., Mamonov, A.V., Zaslavsky, M.: Distance preserving model order
  reduction of graph-{Laplacians} and cluster analysis.
\newblock Journal of Scientific Computing \textbf{90}(1), 32 (2022)

\bibitem{Duykarev}
Dyukarev, Y.M.: Indeterminacy criteria for the {Stieltjes} matrix moment
  problem.
\newblock Mathematical Notes \textbf{75}, 66--82 (2004)

\bibitem{dyukarev2012matrix}
Dyukarev, Y.M., Choque~Rivero, A.E.: The matrix version of {Hamburger’s}
  theorem.
\newblock Mathematical Notes \textbf{91}, 493--499 (2012)

\bibitem{lot2013}
Fenu, C., Martin, D., Reichel, L., Rodriguez, G.: Block {Gauss} and
  anti-{Gauss} quadrature with application to networks.
\newblock SIAM Journal on Matrix Analysis and Applications \textbf{34}(4),
  1655--1684 (2013)

\bibitem{Frommer2009}
Frommer, A.: Monotone convergence of the {Lanczos} approximations to matrix
  functions of {Hermitian} matrices.
\newblock Electronic Transactions on Numerical Analysis \textbf{35}, 118--128
  (2009)

\bibitem{FLSS17}
Frommer, A., Lund, K., Schweitzer, M., Szyld, D.B.: The {R}adau--{L}anczos
  method for matrix functions.
\newblock SIAM Journal on Matrix Analysis and Applications \textbf{38}(3),
  710--732 (2017)

\bibitem{FLS17}
Frommer, A., Lund, K., Szyld, D.B.: Block {K}rylov subspace methods for
  functions of matrices.
\newblock Electron. Trans. Numer. Anal. \textbf{47}, 100--126 (2017)

\bibitem{FLS20}
Frommer, A., Lund, K., Szyld, D.B.: Block {K}rylov subspace methods for
  functions of matrices {II}: Modified block fom.
\newblock SIAM Journal on Matrix Analysis and Applications \textbf{41}(2),
  804--837 (2020)

\bibitem{SimonciniErrorBound}
Frommer, A., Simoncini, V.: Error Bounds for {Lanczos} Approximations of
  Rational Functions of Matrices, p. 203–216.
\newblock Springer-Verlag, Berlin, Heidelberg (2009)

\bibitem{GM10}
Golub, G.H., Meurant, G.: Matrices, Moments and Quadrature with Applications.
\newblock Princeton University Press (2010)

\bibitem{Golub1994}
Golub, G.H., Strako{\v{s}}, Z.: Estimates in quadratic formulas.
\newblock Numerical Algorithms \textbf{8}(2), 241--268 (1994)

\bibitem{GOLUB1977}
Golub, G.H., Underwood, R.: The block {Lanczos} method for computing
  eigenvalues.
\newblock In: J.R. Rice (ed.) Mathematical Software, pp. 361--377. Academic
  Press (1977)

\bibitem{Greenbaum1989}
Greenbaum, A.: Behavior of slightly perturbed {Lanczos} and conjugate-gradient
  recurrences.
\newblock Linear Algebra and its Applications \textbf{113}, 7--63 (1989)

\bibitem{IngermanDtN}
Ingerman, D.: Discrete and continuous {Dirichlet-to-Neumann} maps in the
  layered case.
\newblock SIAM Journal on Mathematical Analysis \textbf{31}(6), 1214--1234
  (2000)

\bibitem{idkGrids}
Ingerman, D., Druskin, V., Knizhnerman, L.: Optimal finite difference grids and
  rational approximations of the square root i. {Elliptic} problems.
\newblock Communications on Pure and Applied Mathematics \textbf{53}(8),
  1039--1066 (2000)

\bibitem{Jones1987TheRM}
Jones, R.: The recursion method with a non-orthogonal basis.
\newblock In: D.G. Pettifor, D.L. Weaire (eds.) The Recursion Method and Its
  Applications, pp. 132--137. Springer Berlin Heidelberg, Berlin, Heidelberg
  (1987)

\bibitem{knizhnerman2002adaptation}
Knizhnerman, L.: Adaptation of the {Lanczos} and {Arnoldi} methods to the
  spectrum, or why the two {Krylov} subspace methods are powerful.
\newblock Chebyshev Digest \textbf{3}(2), 141--164 (2002)

\bibitem{Knizhnerman1996TheSL}
Knizhnerman, L.A.: The simple {Lanczos} procedure: estimates of the error of
  the {G}auss quadrature formula and their applications.
\newblock Computational Mathematics and Mathematical Physics \textbf{36},
  1481--1492 (1996)

\bibitem{snapnets}
Leskovec, J., Krevl, A.: {SNAP Datasets}: {Stanford} large network dataset
  collection.
\newblock {http://snap.stanford.edu/data} (2014)

\bibitem{Liesen2022ComputingTL}
Liesen, J., Nasser, M.M.S., S{\`e}te, O.: Computing the logarithmic capacity of
  compact sets having (infinitely) many components with the charge simulation
  method.
\newblock Numerical Algorithms \textbf{93}, 581--614 (2022).
\newblock \urlprefix\url{https://api.semanticscholar.org/CorpusID:246276213}

\bibitem{Lun18}
Lund, K.: A new block {Krylov} subspace framework with applications to
  functions of matrices acting on multiple vectors.
\newblock Phd thesis, Department of Mathematics, Temple University and
  Fakult\"at Mathematik und Naturwissenschaften der Bergischen Universit\"at
  Wuppertal, Philadelphia, Pennsylvania, USA (2018)

\bibitem{lot2008}
{López Lagomasino}, G., Reichel, L., Wunderlich, L.: Matrices, moments, and
  rational quadrature.
\newblock Linear Algebra and its Applications \textbf{429}(10), 2540--2554
  (2008).
\newblock Special Issue in honor of Richard S. Varga

\bibitem{Meurant2023}
Meurant, G., Tich{\'y}, P.: The behavior of the {Gauss-Radau} upper bound of
  the error norm in {CG}.
\newblock Numerical Algorithms \textbf{94}(2), 847--876 (2023)

\bibitem{MeurantBook}
Meurant, G., Tichý, P.: Error Norm Estimation in the Conjugate Gradient
  Algorithm.
\newblock Society for Industrial and Applied Mathematics, Philadelphia, PA
  (2024)

\bibitem{O'Leary1980}
O'Leary, D.P.: The block conjugate gradient algorithm and related methods.
\newblock Linear Algebra and its Applications \textbf{29}, 293--322 (1980).
\newblock Special Volume Dedicated to Alson S. Householder

\bibitem{RRT16}
Reichel, L., Rodriguez, G., Tang, T.: New block quadrature rules for the
  approximation of matrix functions.
\newblock Linear Algebra and its Applications \textbf{502}, 299--326 (2016).
\newblock Structured Matrices: Theory and Applications

\bibitem{saputra2024adaptive}
Saputra, W., Ambia, J., Torres-Verd{\'\i}n, C., Davydycheva, S., Druskin, V.,
  Zimmerling, J.: Adaptive multidimensional inversion for borehole ultra-deep
  azimuthal resistivity.
\newblock In: SPWLA Annual Logging Symposium, p. D041S013R005. SPWLA (2024)

\bibitem{Chen24}
Xu, Q., Chen, T.: A posteriori error bounds for the block-{L}anczos method for
  matrix function approximation.
\newblock Numerical Algorithms  (2024)

\end{thebibliography}

\appendix

\section{Posterior quadrature error correction for problems with continuous spectrum}\label{sec:PostQuadErr}

{

In this work we mainly target matrices $A$ obtained via (accurate) discretization of operators with continuous spectra corresponding to PDE problems on unbounded domains as detailed in the numerical results. For simplicity, instead of symmetric, matrix valued $ \blF$ obtained from MIMO problems, in this appendix we analyze scalar, single-input single-output (SISO) transfer functions corresponding to $p=1$. For the continuous problems we denote the transfer function as $ \blf$ and denote $\blf_m$ and $\tilde \blf_m$ as the SISO counterparts of, respectively, $\blF_m$ and $\tilde \blF_m$.

Specifically, we assume that the support of the continuous part of the spectral measure of $\blf (z)$ has positive logarithmic capacity \cite{Baker_Graves-Morris_1996,knizhnerman2002adaptation,Liesen2022ComputingTL}, e.g., consists of the union of intervals of the real positive semi axis. 

That allows us to define a positive 
 potential solution on the complex plane corresponding to the 
Green function of the unit charge located at $z=\infty$ with Dirichlet boundary conditions on the support of the continuous components of the spectral measure of $\blf (z)$. This solution plays a fundamental role in the convergence analysis of the Krylov subspace methods, e.g., $g(0)$ is the linear convergence rate of the Conjugate Gradient Method \cite{knizhnerman2002adaptation}.

First, we recall that the Gauss and Gauss-Radau quadratures match $m$ and $m-1$ matrix moments of 
the spectral measure respectively \cite{lot2008}, thus 
\be\label{eq:pade} 
\blf(s)-\tilde\blf_m(s)=O(s^{-2m-1}), \quad \blf(s)-\blf_m(s)=O(s^{-2m}),\quad s\rightarrow \infty .
\ee
Next, we notice, that $\blf_m(s)$ is a $[m-1/m]$ rational function, as follows from \eqref{eq:S-fraction1}. Also, it follows from \eqref{eq:S-fraction2}, that $s\tilde\blf_m(s)$ is a $[m-1/m-1]$ rational function. Thus $\blf_m(s)$ and $s\tilde\blf_m(s)$ are near-diagonal Pad\'e approximants of respectively $\blf_m(s)$ and $s\blf_m(s)$ at $s=\infty$. Both of these functions are Stieltjes, and their Pad\'e approximations converge uniformly in any compact set not including the support of the spectral measure. This together with Proposition~\ref{prop:03} allows
an adoption of  Stahl's Pad\'e 
convergence Theorem \cite[Theorem 6.6.9]{Baker_Graves-Morris_1996} based on 
potential theory, to this specific case and it can be stated as follows.

\begin{proposition}\label{lem:stahl}
	 
	 $\forall [a,b]$, with $0<a<b$ and $\forall \epsilon>0$ there exists an $m$ such that 
		\begin{equation}
		(e^{ -g(s)}-\epsilon)^{2m-2}<\tilde{ \blf}_m(s) -\blf(s)
		< (e^{ -g(s)}+\epsilon)^{2m-2}, \label{eq:b1}
		\end{equation}
		and
		\begin{equation} 
		(e^{ -g(s)}-\epsilon)^{2m-1}<\blf(s)-\blf_m(s) <(e^{ -g(s)}+\epsilon)^{2m-1}.\label{eq:b2}
		\end{equation}	
\end{proposition}

Proposition~\ref{lem:stahl} guides us to the error correction formula $\hat f_m=\frac 1 2 (f_m+\tilde f_{m+1})$, which in the $p>1$ setting is stated in Equation~(\ref{eq:Extrapol1}). 


Subtracting \eqref{eq:b2} from \eqref{eq:b1} we obtain the following estimate:
\begin{proposition}\label{prop:corr}
	For $p=1$, small $g(s)$ and thus $ 1- e^{-g(s)}\approx g(s)$ the error of the averaged $\hat f_m(s)$ behaves as
	\[|f(s)-\widehat \blf_m(s)|\approx [g(s)+o(1)] |\blf(s)- \blf_m(s)|,\]
 for large $m$.
\end{proposition}}

The potential solution $g(s)$ is small when $s$ is close to the boundary of continuous components of the spectra, i.e., in the case of the spectral interval $[0, \theta_{max}]$, for small positive $s$. Since $2g(s)$ is the linear rate of convergence of $\blf_m$, slower convergence of $\blf_m$ leads to better convergence of the  averaged $\widehat f_m$.

\begin{remark}\label{rem:duality}
	In the context of inverse problem theory, certain PDEs exhibit a significant symmetry structure known as duality, as illustrated in \cite{PhysRevLett.65.325}. This duality implies that the Neumann-to-Dirichlet map for the primary problem can be linearly scaled to correspond to the Dirichlet-to-Neumann map for the dual problem. The transfer function $\blF(s)$, when $B$ is localized at a (partial) boundary of the computational domain, can be interpreted as a discretization of the partial Neumann-to-Dirichlet map and, similarly, can be scaled to represent the Dirichlet-to-Neumann map of the dual problem, denoted as $\blF^{-1}(s)$. Following this logic, it can be demonstrated that $\tilde\blF_m(s)$ can be transformed into $\blF_{m+1}(s)$ for the dual problem. While a rigorous analysis of this phenomenon within the framework of error  correction via averaging is warranted, we leverage this duality intuitively to refine the averaging formula. To adhere to the duality structure, the formula should remain invariant under the substitution of $\blF^{-1}$ for $\blF$, a property that \eqref{eq:Extrapol1} does not satisfy. Hence, we propose a similar, but dual-invariant, formula that satisfies Proposition~\ref{prop:corr}, (see Equation~(\ref{eq:Extrapol2})),
	\begin{equation}\label{eq:Extrapol2_A}
	\widecheck \blF_m(s):= \exp\left[\frac 1 2 [\log(\overline{\blF_m}(s))+\log(\widehat \blF_{m}(s))]\right] 
	\end{equation}
	with 
	\begin{equation}
	\overline{\blF_m}^{-1}:=\frac 1 2 \left [(\blF_m(s))^{-1}+(\tilde \blF_{m+1}(s))^{-1}\right].
	\end{equation}
 { Intuitive reasoning in this remark is limited to 2D problems, and interestingly enough our numerical experiments show advantage of such an averaging compared to the simple arithmetic one only for 2D models as shown in section~\ref{sec:subsubsection2D}. We hope to study this phenomenon more rigorously in our future work.}

\end{remark}

\section{Random enrichment of initial block}\label{sec:Enrichment}
Adding columns to the initial block $B$ can reduce the number of block iterations only if 
 the Krylov subspaces generated by the columns have a significant intersection. This is not the case for the problems with localized inputs/outputs considered in section~\ref{sec:PDE}. Such problems however are important in remote sensing applications. To improve convergence for such a class of problems we experiment with adding random column vectors to $B$ for the problem considered in section~\ref{sec:subsubsection2D}. We used an enriched block $B'=[B,R]$, where $R\in \RR^{n\times p'}$ is a block of random vectors. The Gauss and Gauss-Radau quadratures using $\mathcal{K}_i(A,B')$ are computed and the leading $p\times p$ block of the computed approximation to ${B'}^T\phi(A,s)B'$ is used to approximate ${B}^T\phi(A,s)B$.

In Figure~\ref{fig:RandVec} we illustrate the effect of adding random vectors to $B$ on the convergence of the method for the shift $s=10^{-2}$ for the 2D diffusion experiments from section~\ref{sec:subsubsection2D}. The vectors added have independent, identically distributed entries chosen from the normal distribution $\mathcal{N}(0,1)$. Since the block Lanczos method is a memory-bandwidth limited algorithm on most modern CPU and GPU architectures, adding extra vectors to $B$ comes at a small increase in computational cost per block Lanczos iteration, yet it greatly affects convergence rates. In the solid black curve only the 3 transducers are used in $B$, adding one random vector to $B$ increases the convergence rate as shown by the dashed line, and adding two additional random vectors to $B$ accelerated convergence further if expressed in terms of block iterations; see, e.g., \cite{O'Leary1980} for a theoretical analysis. To reach an error of $10^{-6}$ the block Lanczos iteration with no added vectors needed $m=202$ iterations (solid black line), with one added random vector in needed $m=169$ iterations (dashed black line), with 3 added vectors it required $m=145$ iterations (dotted black line) and with 12 added vectors $m=86$ iterations (dash-dotted black line). With averaging only $m=129$ iterations were needed for the case of 3 added random vectors. 

The acceleration of averaged quadrature rules is even greater at higher error levels, which can be qualitatively explained as follows. For problems with an (approximately) continuous spectrum, adding vectors to the subspace does not improve the linear convergence rate, resulting in sub-linear acceleration. This effect is amplified by the averaging , which, according to Proposition~\ref{prop:corr} is more significant at the low linear convergence rates observed in the final stage of sub-linear convergence. This phenomenon can be practically important for applications to inverse problems that tolerate low-precision computations due to measurement errors.

It is important to note that adding random starting vectors leads to an increase in the overall size of the matrix $T_m$ due to the enlargement of the block size, despite the acceleration in convergence. Thus any reduction in computation speed comes from sparse matrix-matrix multiplications $AB$ which are memory-bandwidth limited for small $p$.

\begin{figure}[h!b]
	\centering
		\centering
		\includegraphics[width = 0.75\linewidth]{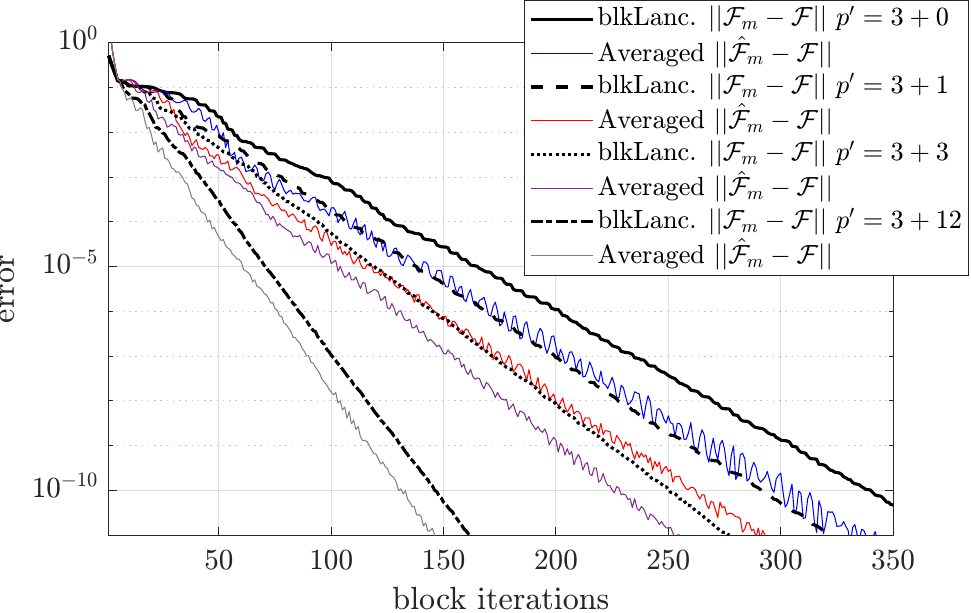}		
		\caption{Acceleration of convergence by adding random vectors to the starting { block }vector $B$ for shift $s=10^{-2}$.}
		\label{fig:RandVec}
\end{figure}

\end{document}